\documentclass[11pt]{amsart}
\usepackage{amssymb}
\usepackage{verbatim}
\usepackage{hyperref}
\usepackage{color}

\usepackage{scalerel}
\newcommand\vwidehat[1]{\arraycolsep=0pt\relax%
\begin{array}{c}
\stretchto{
  \scaleto{
    \scalerel*[\widthof{\ensuremath{#1}}]{\kern-.5pt\bigwedge\kern-.5pt}
    {\rule[-\textheight/2]{1ex}{\textheight}} 
  }{\textheight} %
}{0.5ex}\\           
#1\\                 
\rule{-1ex}{0ex}
\end{array}
}

\usepackage{mathabx}

\newtheorem{theorem}{Theorem}    
\newtheorem{proposition}[theorem]{Proposition}

\newtheorem{corollary}[theorem]{Corollary}
\newtheorem{lemma}[theorem]{Lemma}

\theoremstyle{definition}
\newtheorem{definition}{Definition}
\newtheorem{question}[definition]{Question}

\numberwithin{theorem}{section}
\numberwithin{definition}{section}
\numberwithin{equation}{section}

\title[A constrained optimization problem]{A constrained optimization problem for\\  the Fourier transform: Existence}

\author{Dominique Maldague}
\address{
        Dominique Maldague\\
        Department of Mathematics\\
        University of California \\
        Berkeley, CA 94720-3840, USA}
\email{dmaldague@berkeley.edu}

\date{March 10, 2019.}

\newcommand{\R}{\mathbb R}

\newcommand{\Z}{\mathbb Z}
\newcommand{\C}{\mathbb C}

\newcommand{\T}{\mathbb T}
\newcommand{\F}{\mathcal F}

\newcommand{\B}{\mathbb B}

\newcommand{\g}{\gamma}
\newcommand{\p}{\varphi}

\renewcommand{\l}{\ell}
\renewcommand{\a}{\alpha}

\newcommand{\mc}{\mathcal}
\newcommand{\mb}{\mathbb}
\newcommand{\f}{\mathfrak}

\begin{document}

\setcounter{tocdepth}{1}

\begin{abstract}
Among functions majorized by indicator functions of sets with measure one, which functions have maximal  Fourier transforms in the $L^q$ norm? We partially prove the existence of such functions using techniques from additive combinatorics  to establish a conditional precompactness for maximizing sequences. 

\end{abstract}

\maketitle

\tableofcontents

\section{Introduction}
Define the Fourier transform as $ \mc{F}(f)(\xi)=\widehat{f}(\xi)=\int_{\R^d}e^{- 2\pi i x\cdot \xi}f(x)dx$ for a function $f:\R^d\to\C$. The Fourier transform is a contraction from 
$L^1(\R^d)$ to $L^\infty(\R^d)$ and is unitary on $L^2(\R^d)$. Interpolation gives the Hausdorff-Young inequality $\|\widehat{f}\|_q\le \|f\|_p$ where $p\in (1,2)$, $1=\frac{1}{p}+\frac{1}{q}$. In \cite{beckner}, Beckner proved the sharp Hausdorff-Young inequality 

\begin{align} \|\widehat{f}\|_q\le {\bf{C}}_q^d\|f\|_p \label{maineq} \end{align}

where ${\bf{C}}_q=p^{1/2p}q^{-1/2q}$. In 1990, Lieb proved that Gaussians are the \emph{only} maximizers of (\ref{maineq}), meaning that $\|\widehat{f}\|_q/\|f\|_p={\bf{C}}_q^d$ if and only if $f=c\exp(-Q(x,x)+v\cdot x)$ where $Q$ is a positive definite real quadratic form, $v\in\C^d$ and $c\in\C$. In 2014, Christ established a sharpened Hausdorff-Young inequality by bounding $\|\widehat{f}\|_q - {\bf{C}}_p^d \|f\|_p$ by a negative multiple of an $L^p$ distance function of $f$ to the Gaussians.

In \cite{c2}, Christ made partial progress proving the existence of maximizers for the ratio $\|\widehat{1_E}\|_q/|E|^{1/p}$ where $E\subset\R^d$ is a positive Lebesgue measure set. Building on the work of Burchard in \cite{burchard}, Christ identified maximizing  sets to be ellipsoids for exponents $q\ge 4$ sufficiently close to even integers \cite{c2}. The author correspondingly identified all maximizers for the inequality studied in this paper specialized to exponents $q\ge 4$ close to even integers in \cite{me!}.

Another variant of the Hausdorff-Young inequality replaces indicator functions by bounded multiples and modifies the functional as follows.
For $d\ge 1$, $q\in(2,\infty)$, and $p=q'$, we consider the inequality

\begin{equation}\label{eqn:main}
\|\widehat{f}\|_q\le {\bf{B}}_{q,d}|E|^{1/p}
\end{equation}
and define the quantities 
\begin{align}
\Psi_q(E):=\sup_{|f|\prec E} \frac{\|\widehat{f}\|_{q}}{\|1_E\|_p} \label{eq2}\\
{\bf{B}}_{q,d}:=\sup_{E}\Psi_q(E) \label{thisone}
\end{align}
where $|f|\prec E$ means $|f|\le 1_E$ and the supremum is taken over all Lebesgue measurable sets $E\subset\R^d$ with positive, finite Lebesgue measures. This quantity ${\bf{B}}_{q,d}$ is less than ${\bf{C}}_p^d$ by their definitions. The supremum (\ref{thisone}) is equal to 

\[   \sup_{f\in L(p,1)}\frac{\|\widehat{f}\|_q}{\|f\|_{\mc{L}}}\quad\quad\text{where}\quad\quad \|f\|_{\mc{L}}=\inf\{\|a\|_{\ell^1}\,:|f|=\sum_n a_n|E_n|^{-1/p}1_{E_n},\, a_n>0,|E_n|<\infty\} .\]

We prove this equivalence in Proposition \ref{Lorentz} in \textsection \ref{Lorentzdiscussion}. Lorentz spaces are a result of real interpolation between $L^p$ spaces. Since the quasinorm $\|\cdot\|_\mc{L}$ induces the standard topology on the Lorentz space $L(p,1)$, this is a natural quantity to study.

Christ used continuum versions of theorems of Balog-Szemer\'{e}di and Fre\u{\i}man from additive combinatorics to understand the underlying structure of functions with nearly optimal ratio $\|\widehat{f}\|_q/\|f\|_p$ in \cite{c1} and for sets $E$ with nearly optimal ration $\|\widehat{1_E}\|_q/\|1_E\|_p$ in \cite{c2}. We use similar techniques in this paper to (conditionally) prove the existence of extremizers for (\ref{maineq}) via a precompactness argument for extremizing sequences, presented in the following theorem. The theorem is conditional on an affirmative answer to a technical question, which is presented as Question \ref{dual near-ext=slice} in \textsection{\ref{product}}. 

\begin{theorem}\label{precompactness} Suppose that the claim in Question \ref{dual near-ext=slice} holds. Let $d\ge1$ and $q\in(2,\infty)$, $p=q'$. Let $(E_\nu)$ be a sequence of Lebesgue measurable subsets of $\R^d$ with $|E_\nu|\in\R^+$ and let $f_\nu$ be Lebesgue measurable functions on $\R^d$ satisfying $|f_\nu|\le 1_{E_\nu}$. Suppose that $\lim_{\nu\to\infty}|E_\nu|^{-1/p}\|\widehat{f_\nu}\|_{q}={\bf{B}}_{q,d}$. Then there exists a subsequence of indices $\nu_k$, a Lebesgue measurable set $E\subset\R^d$ with $0<|E|<\infty$, a Lebesgue measurable function $f$ on $\R^d$ satisfying $|f|\le 1_E$, a sequence  $(T_\nu)$ of affine automorphisms of $\R^d$, and a sequence of vectors $v_\nu\in\R^d$ such that  
\[ \lim_{k\to\infty}\|e^{-2\pi i v_{\nu_k}\cdot x}f_{\nu_k}\circ T_{\nu_k}^{-1}-f\|_p=0\quad\text{and}\quad \lim_{k\to\infty}|T_{\nu_k}(E_{\nu_k})\Delta E|=0. \]
\end{theorem}
The conditional existence of maximizers is a direct consequence. A simplified outline of the argument is as follows. 
\begin{enumerate}
\item Begin by proving basic principles of concentration compactness: ``no slacking" and ``cooperation" (see \textsection{\ref{conccpct}}). 

\item If $|f|\le 1_{E}$ with $|E|=1$ satisfies $\|\widehat{f}\|_{q}\ge \eta$ for $\eta>0$, then $f$ satisfies a related Young's convolution inequality: for appropriate $\g,r,s$, $\||f|^\g*|f|^\g\|_r\ge \eta^s$. 

\item By continuum analogues of theorems of Balog-Szemer\'edi and Fre\u{\i}man, $|f|\le 1_E$ with $|E|=1$ satisfying $\||f|^\g*|f|^\g\|_r\ge \eta^s$ must place a portion of its $L^p$ mass on a continuum multiprogression of controlled rank and Lebesgue measure. 

\item Combine concentration compactness  principles with the specific additive structure we have from the relation to Young's convolution inequality to conclude that a function $|f|\le 1_E$ satisfying $\|\widehat{f}\|_q\ge (1-\delta){\bf{B}}_{q,d}|E|^{1/p}$ for small $\delta>0$ is mostly supported on a multiprogression of controlled rank and size.  

\item By precomposing a near-extremizer with an affine transformation $\mc{T}$, we can change variables to guarantee that the continuum multiprogression is mostly contained in $\Z^d\times[-\delta,\delta]^d$. We must guarantee that the Jacobian of $\mc{T}$ is bounded below since otherwise we could trivially collapse any bounded set to a small neighborhood of the origin. 

\item The Fourier transform of a function living on $\Z^d\times [-\delta,\delta]^d$ decomposes into a discrete and a continuous Fourier transform, and a near-extremizer for (\ref{eqn:main}) must be a near-extremizer of each step of the decomposition. Since near-extremizers of the discrete Fourier transform must mostly be supported on a single $n\in\Z^d$, this gives extra structure. We prove that the only multiprogression structure which is favorable at each step of the decomposition is one mostly contained in a single convex set $[-\delta,\delta]$. 

\item If $|f|\le 1_E$ is a near-extremizer, then $\widehat{f}|\widehat{f}|^{q-2}$ is a near-extremizer of a related dual inequality (see \textsection{\ref{dualsect}}). The above reasoning may also be carried out for this dual inequality, except for step (6), which may or may not be true in the dual setting. If step (6) holds in the dual setting, we conclude that a significant portion of the $L^p$ mass of $f$ and $\widehat{f}$ must be localized to ellipsoids (or other convex sets) of controlled size.

\item Via a composition with an affine transformation and modulation by a character, we can assume that $f$ and $\widehat{f}$ are localized (respectively) to the unit ball $\B$ and and ellipsoid $\mc{E}$ centered at the origin. We prove a reversed uncertainty bound: $|\mc{E}||\B|\le C$ and furthermore $\mc{E}\subset C\B$ for an appropriate $C>0$. 

\item It follows that for any sequence of function $|f_\nu|\le 1_{E_\nu}$ with $|E_\nu|\in\R^+$ and\newline  $\|\widehat{f_\nu}\|_q|E_\nu|^{-1/p}\to\ {\bf{B}}_{q,d}$, after $(f_\nu,E_\nu)$ is renormalized to $(F_\nu,A_\nu)$ by appropriate symmetries of the inequality, we have weakly convergent subsequences of $F_\nu$ and $1_{A_\nu}$. Finally, we get $L^p$ convergence via a convexity argument involving the $\|\cdot\|_{\mc{L}}$ norm. 

\end{enumerate}

This material is based upon work supported by the National Science Foundation Graduate Research Fellowship under Grant No. DGE 1106400.

\section{Results in terms of the Lorentz space $L(p,1)$\label{Lorentzdiscussion}}

There are many quasinorms which induce the same topology on $L(p,q)$ spaces. For the special case of $p>1$ and $q=1$, we will show that our extremization problem can be phrased using various quasinorms (and one norm defined by Calder\'{o}n) on $L(p,1)$. Let ${\bf{B}}_{q,d}$ be as before.

\begin{definition} Let $d\ge 1$. Define $\|f\|_{\mc{L}}$ for a measurable function $f:\R^d\to\R$ by
\[ \|f\|_{\mc{L}}=\inf\{\|(a_n)\|_{\l^1}:\,|f|=\sum_n a_n|E_n|^{-1/p}1_{E_n},\,a_n\ge 0,\,|E_n|<\infty\}.\]

\end{definition}

The following definitions \ref{deff*}, \ref{defp1*}, and \ref{defp1} are from Chapter V, \textsection{}3 in \cite{steinweiss}. 

\begin{definition}\label{deff*} Let $d\ge 1$. Define $f^*$ for $t>0$ by 
\[ f^*(t)=\inf\{r: |\{x:|f(x)|>r\}|\le t\}. \]

\end{definition}

\begin{definition}\label{defp1*}  Let $d\ge 1$, $1\le p< \infty$, $q$ the conjugate of $p$. Define $\|f\|^*_{p1}$ for a measurable function $f:\R^d\to\R$ by  
\[ \|f\|_{p1}^*= \frac{1}{p}\int_0^\infty t^{-1/q}f^*(t)dt.  \]
\end{definition}

\begin{definition}\label{defp1}   Let $d\ge 1$, $1\le p< \infty$, $q$ the conjugate of $p$. Define $\|f\|_{p1}$ for a measurable function $f:\R^d\to\R$ by 
\[ \|f\|_{p1}= \frac{1}{p}\int_0^\infty t^{-1/q-1}\int_0^tf^*(u)dudt. \]
\end{definition}

\begin{definition} Let $d\ge 1$, $1\le p<\infty$. The space $L(p,1)$ is defined as all measurable functions $f:\R^d\to\C$ satisfying $ \|f\|_{p1}^*<\infty$.
\end{definition}

See the appendix for the relationships between $\|\cdot\|_{\mc{L}}$, $\|\cdot\|_{p1}^*$, and $\|\cdot\|_{p1}$, and that they generate  the same topology on $L(p,1)$. In particular, it is proved that $\|f\|_{\mc{L}}=\|f\|_{p1}^*$ for all measurable $f:\R^d\to\C$ (where one quantity is infinite if and only if the other quantity is as well).

\begin{proposition} \label{Lorentz} For $d\ge 1$, $q\in(2,\infty)$, and $p$ the dual exponent to $q$, 
\[ {\bf{B}}_{q,d}=\sup_{f\in{L^p}}\frac{\|\widehat{f}\|_q}{\|f\|_{\mc{L}}}.  \]
\end{proposition}

\begin{proof} Let $|f|=\sum_n a_n |E_n|^{-1/p}1_{E_n}$ where $a_n\ge 0$ and $|E_n|<\infty$. Then $\|f\|_p\le \sum_n a_n$, so $\|f\|_p\le \|f\|_{\mc{L}}$. By the Hausdorff-Young inequality, the constant  $A_{\mc{L}}$ defined by 
\begin{equation}
    A_{\mc{L}}:=\sup_{f\in L^p}\frac{\|\widehat{f}\|_q}{\|f\|_{\mc{L}}} \label{eq:dualmain}
\end{equation}
is finite. 

We want to show that $ {\bf{B}}_{q,d}:=  \sup_{|E|<\infty}\sup_{|f|\prec 1_{E}}\frac{\|\widehat{f}\|_q}{|E|^{1/p}}=\sup_{f\in{L^p}}\frac{\|\widehat{f}\|_q}{\|f\|_{\mc{L}}}=:A_{\mc{L}} $.

If $|f|=\sum a_n|E_n|^{-1/p}1_{E_n}$ with $a_n\ge 0$, $|E_n|<\infty$,  then 
\[ \frac{\|\widehat{f}\|_q}{\sum |a_n|}\le \frac{\sum|a_n||E_n|^{-1/p}\|\widehat{1_{E_n}}\|_q}{\sum|a_n|}\le {\bf{B}}_{q,d} ,\]
so $A_{\mc{L}}\le {\bf{B}}_{q,d}$.

For the other direction, since simple functions are dense in $L^p(\R^d)$, it suffices to consider $f=\sum a_n1_{A_n}$ where $A_n$ are disjoint and $f$ is majorized by the indicator of a Lebesgue measurable set $E$ of size one. Then $\sum |a_n||A_n|^{1/p}\le |E|^{1/p}\sum |a_n|=|E|^{1/p}\|f\|_1\le |E|^{\frac{1}{p}+1}=1$. Rearranged, this means  
\[ \|\widehat{f}\|_q=\frac{\|\widehat{f}\|_q}{|E|^{1/p}}\le \frac{\|\widehat{f}\|_q}{\sum|a_n||A_n|^{1/p}} \le \frac{\|\widehat{f}\|_q}{\|f\|_{\mc{L}}},\]
so ${\bf{B}}_{q,d}\le A_{\mc{L}}$. 

\end{proof}

\begin{lemma}\label{extstruct} If $f\in L(p,1)$ satisfies ${\bf{B}}_{q,d}=\|f\|_{\mc{L}}^{-1}\|\widehat{f}\|_q$, then 
\[ f=ae^{i\p}1_E\]
for some scalar $a\in \R^+$, Lebesgue measurable function $\p:\R^d\to\R$, and a Lebesgue measurable set $E$ of finite measure. 
\end{lemma}

\begin{proof} By Lemma \ref{equivnorms}, we also have that ${\bf{B}}_{q,d}=(\|f\|_{p1}^*)^{-1}\|\widehat{f}\|_q$. Let $E=\{(y,s):|f(y)|>s \} $. Let $e^{i\p}=f/|f|$ so we can use the layer cake representation
\[    f(x)=e^{i\p(x)}\int_0^\infty1_{E}(x,t)dt   .  \]
Then 
\begin{align*}  
\|\widehat{f}\|_q &= \left\|\left(\int_0^\infty e^{i\p(x)}1_{E}(x,t)dt\right)^{\widehat{\,}}\,\,\right\|_q = \left\| \int_0^\infty \widehat{e^{i\p}1_{E}}(\xi,t)dt\right\|_q  \\ 
&\le \int_0^\infty \|\widehat{e^{i\p}1_E}(\xi,t)\|_q dt \\
&\le \int_0^\infty {\bf{B}}_{q,d}\left|\{x:|f(x)|>t\}\right|^{1/p}dt \\
&={\bf{B}}_{q,d} \int_0^\infty \int_0^{|\{x:|f(x)|>t\}|}\frac{1}{p}u^{-1/q}du dt \\
&={\bf{B}}_{q,d} \int_0^{|\{x:|f(x)|>0\}|} \int_0^{f^*(u)}\frac{1}{p}u^{-1/q} dt du={\bf{B}}_{q,d}\|f\|_{p1}^*.
\end{align*} 

Since ${\bf{B}}_{q,d}=(\|f\|_{p1}^*)^{-1}\|\widehat{f}\|_q$, the above sequence of inequalities are actually equalities. Equality in the Minkowski integral inequality implies that for a.e. $(\xi,t)\in \R^d\times\R^+$,
\[ \widehat{e^{i\p}1_E}(\xi,t)=h(\xi)g(t)\]
for some measurable functions $h,g$. 
Since $e^{i\p}1_E(x,t)\in L^{2}$, in particular, $h$ and $\widecheck{h}$ in $L^2$. 
\[ 1_E(x,t)=e^{-i\p(x)}\widecheck{h}(x)g(t). \]
But then for every $(x,t)$ satisfying $|f(x)|>t$, we have 
\[ e^{-i\p(x)}\widecheck{h}(x)g(t)=1. \]
Suppose $|f(x)|>|f(y)|>0$. Then for all $0\le t<f(y)$, 
\[ e^{-i\p(x)}\widecheck{h}(x)=g(t)^{-1}=e^{i\p(y)}\widecheck{h(y)} ,\]
which is a contradiction unless $|f(x)|$ is constant on its support. Thus $f$ takes the form $ae^{i\p}1_S$ where $S=\text{supp }f$ and $a\in \R^+$.  
\end{proof}

The existence corollary to Theorem \ref{precompactness} in terms of Lorentz norms is 
\begin{corollary}\label{maincor}Let $d\ge1$, $p\in(1,2)$, $q$ the conjugate exponent of $p$. First, we have $\frac{\textbf{B}_{q,d}}{q}=\sup_{0\not=g\in L(p,1)}\|g\|_{p1}^{-1}\|\widehat{g}\|_{q}$. Second, if $f\in L(p,1)$ satisfies $\frac{{\bf{B}}_{q,d}}{q}=\|f\|_{p1}^{-1}\|\widehat{f}\|_q$, then 
\[ f=ae^{i\p}1_E\]
for some scalar $a\in \R^+$, Lebesgue measurable function $\p:\R^d\to\R$, and a Lebesgue measurable set $E$ of finite measure. 
\end{corollary}

See \textsection{\ref{precompactnesssect}} for the proof of the Corollary \ref{maincor}.

\section{The dual inequality \label{dualsect}}

Recall the definition of the optimal constant ${\bf{B}}_{q,d}$
\[ {\bf{B}}_{q,d}=\sup_{0<|E|<\infty}\sup_{|f|\le 1_E}\frac{\|\widehat{f}\|_q}{|E|^{1/p}}. \]

By exploiting $L^p$ duality and Plancherel's theorem, we also have the expressions: 
\begin{align*}  
{\bf{B}}_{q,d}& =\sup_{|E|<\infty}\sup_{|f|\le 1_E}\sup_{\|g\|_p\le 1}\frac{|\langle \widehat{f},g\rangle |}{|E|^{1/p}}=\sup_{|E|<\infty}\sup_{\|g\|_p\le 1}\frac{\langle 1_E,|\widehat{g}|\rangle |}{|E|^{1/p}},
\end{align*} 
the last of which motivates the following definition.

\begin{definition} \label{defdual} Let $d\ge 1$ and $q\in[1,\infty)$, and $p$ be the conjugate exponent to $q$. Define the norm $\|\cdot\|_{q,*}$ of a function $g\in L^q(\R^d)$ to be 
\[ \|g\|_{q,*}=\sup_{0<|E|<\infty}|E|^{-1/p}\int_E|g| \]
where the supremum is taken over Lebesgue measurable subsets $E\subset\R^d$ of positive, finite measure.  
\end{definition}

Note that by H\"{o}lder's inequality, if $g\in L^q$, then $\|g\|_{q,*}\le \|g\|_q<\infty$. Thus for $f\in L^p$ with $p\in(1,2)$ and $q$ the conjugate exponent, 
\[ \|\widehat{f}\|_{q,*}\le \|f\|_{p}\]
is a corollary of the Hausdorff-Young inequality.

\begin{proposition} \label{dualineq} For $d\ge 1$, $q\in(2,\infty)$, and $p$ the dual exponent to $q$, 
\[ {\bf{B}}_{q,d}=\sup_{\|f\|_p\le 1}\|\widehat{f}\|_{q,*} . \]
Furthermore, if $|f|\le 1_E$, $|E|<\infty$ satisfies $\|\widehat{f}\|_q\ge (1-\delta){\bf{B}}_{q,d}|E|^{1/p}$ for some $\delta>0$, then
\[   \|(|\widehat{f}|^{q-2}\widehat{f})^{\widecheck{\,\,}}\|_{q,*}\ge (1-\delta)^q{\bf{B}}_{q,d}\||\widehat{f}|^{q-2}\widehat{f}\|_p.     \]
\end{proposition}

\begin{proof} Let $f\in L^p(\R^d)$. Consider a Lebesgue measurable set $E\subset \R^d$ of finite measure such that $|\widehat{f}|\not=0$ a.e. on $E$ and write $\widehat{f}=e^{-i\p}|\widehat{f}|$ for a real-valued phase function $\p$ equal to $0$ off of the support of $\widehat{f}$. Using Plancherel's theorem and H\"{o}lder's inequality, we then have  

\begin{align*}
    |E|^{-1/p}\int_E |\widehat{f}|&=|E|^{-1/p}\int 1_Ee^{i\p}\widehat{f} = |E|^{-1/p}\int \widehat{e^{i\p}1_E} {f}\\
    &\le |E|^{-1/p} \|\widehat{e^{i\p}1_E}\|_q\|f\|_p  \\
    &\le {\bf{B}}_{q,d} \|f\|_p,
\end{align*}
so that $\underset{{\|f\|_p\le 1}}\sup \|\widehat{f}\|_{q,*}\le {\bf{B}}_{q,d}$.

Now suppose that for $|f|\le 1_E$, $|E|<\infty$, and $\delta>0$ we have $\|\widehat{f}\|_q\ge (1-\delta){\bf{B}}_{q,d}|E|^{1/p}$. Then $|\widehat{f}|^{q-2}\widehat{f}\in L^p$ since $\||\widehat{f}|^{q-2}\widehat{f}\|_p^p==\|\widehat{f}\|_q^q$. Then 
\begin{align*}
\int_{E}\left|\left({|\widehat{f}|^{q-2}\widehat{f}}\right)^{\widecheck{\,\,}} \right|&\ge  \int |{f}|\left|\left({|\widehat{f}|^{q-2}\widehat{f}}\right)^{\widecheck{\,\,}} \right| \\
&\ge \left|\int {f}\,\overline{\left({|\widehat{f}|^{q-2}\widehat{f}}\right)^{\widecheck{\,\,}}} \right| \\
&= \left|\int \widehat{f}\overline{|\widehat{f}|^{q-2}\widehat{f}} \right|=\int |\widehat{f}|^{q}\\
&\ge (1-\delta)^q{\bf{B}}^q_{q,d} |E|^{q/p}.
\end{align*}

Rearranging and using that $\||\widehat{f}|^{q-2}\widehat{f}\|_p=\|\widehat{f}\|_q^{q/p}\le {\bf{B}}_{q,d}^{q/p}|E|^{q/p^2}$, 

\[ |E|^{-1/p} \int_{-E}\left|\left({|\widehat{f}|^{q-2}\widehat{f}}\right)^{\widehat{\,\,}} \right|\ge (1-\delta)^q{\bf{B}}_{q,d}^{q-q/p}\||\widehat{f}|^{q-2}\widehat{f}\|_p = (1-\delta)^q{\bf{B}}_{q,d}\||\widehat{f}|^{q-2}\widehat{f}\|_p,    \]

proving that we can find $g\in L^p$ such that $\|\widehat{g}\|_{q,*}\|g\|_p^{-1}$ is arbitrarily close to ${\bf{B}}_{q,d}$. 

\end{proof}

For $q\in (2,\infty)$ and $p$ the conjugate exponent of $q$, the inequality 
\begin{align}  \|\widehat{g}\|_{q,*} \le {\bf{B}}_{q,d}\|g\|_p .\label{eq:dual} \end{align} 
is amenable to the same analysis as our main inequality $(\ref{eqn:main})$, and each lemma we prove about (\ref{eqn:main}) will have an analogue for this dual inequality.

\section{Quasi-extremal principles \label{conccpct} }

We establish the quasi-extremal principles ``no slacking" and ``cooperation". No slacking guarantees that a near-extremizer is a combination of small parts which must be quasi-extremizers. Cooperation guarantees that these small parts work together in a compatible way (e.g. have nontrivial intersection of supports).

\begin{definition} Let $d\ge 1$, $q\in(2,\infty)$ and $p=q'$. A nonzero function $f$ satisfying $|f|\le 1_E\in L^p$ is a $\delta$-quasi-extremizer for (\ref{eqn:main}) if 
\[ \|\widehat{f}\|_q\ge (1-\delta){\bf{B}}_{q,d}|E|^{1/p}. \]
\end{definition}
By a quasi-extremizer, we mean a $\delta$-quasi-extremizer for some small $\delta>0$. 

\subsection{No slacking}

\begin{lemma}[No slacking] \label{noslacking} For any $p,q\in(1,\infty)$ there exist $c,C_0<\infty$ with the following property. Let $\delta>0$, $|E|<\infty$, $|f|\le 1_E$. Suppose that 
\[ \|\widehat{f}\|_q\ge (1-\delta){\bf{B}}_{q,d}|E|^{1/p}.\]
Suppose that $f=g+h$ where $g=1_Af,\,h=1_Bf$, and $A\cap B=\empty$, and that
\[ |B|\ge C_0\delta |E|.\]
Then 
\[ \|\widehat{h}\|_q\ge c\delta|E|^{1/p}.\]
\end{lemma}
\begin{proof} There exists $C<\infty$ such that for any $G,H\in L^q$,
\[ \|G+H\|_q^q\le \|G\|_q^q+C\|G\|_q^{q-1}\|H\|_q+C\|H\|_q^q. \]
Consequently,
\begin{align*}
    \|\widehat{g+h}\|_q^q&\le \|\widehat{g}\|_q^q+C\|\widehat{g}\|_q^{q-1}\|\widehat{h}\|_q+C\|\widehat{h}\|_q^q \\
    &\le {\bf{B}}_{q,d}^q |A|^{q/p}+C{\bf{B}}_{q,d}^{q-1}|A|^{(q-1)/p}\|\widehat{h}\|_q+C\|\widehat{h}\|_q^q.
\end{align*}
On the other hand, $|E|=|A|+|B|$. Without loss of generality, assume $|E|=1$, so that $|A|,|B|\le1$. Thus 
\begin{align*}
    (1-\delta)^q&\le \frac{\|\widehat{f}\|_q^q}{{\bf{B}}_{q,d}^q|E|^{q/p}}=\frac{\|\widehat{f}\|_q^q}{{\bf{B}}_{q,d}^q} \\
    &\le {\bf{B}}_{q,d}^{-q}({\bf{B}}_{q,d}^q |A|^{q/p}+C{\bf{B}}_{q,d}^{q-1}|A|^{(q-1)/p}\|\widehat{h}\|_q+C\|\widehat{h}\|_q^q) \\
    &= |A|^{q/p}+C{\bf{B}}_{q,d}^{-1}|A|^{(q-1)/p}\|\widehat{h}\|_q+C{\bf{B}}_{q,d}^{-q}\|\widehat{h}\|_q^q \\
    &\le (1-|B|)^{q/p}+C{\bf{B}}_{q,d}^{-1}|A|^{(q-1)/p}\|\widehat{h}\|_q+C{\bf{B}}_{q,d}^{-1}\|\widehat{h}\|_q \\
    &\le 1-c_p|B|+2C{\bf{B}}_{q,d}^{-1}\|\widehat{h}\|_q .
\end{align*}
Then we have 
\begin{align*}
    2C{\bf{B}}_{q,d}^{-1}\|\widehat{h}\|_q&\ge  (1-\delta)^q-1+c_p|B| \\
    &\ge 1-O(\delta)-1+|B|\\
    &\ge |B|-O(\delta) \\
    &\ge C_0^p\delta -O(\delta) \\
    &\ge c\delta 
\end{align*}
provided $C_0$ is large enough. 
\end{proof}

\begin{lemma}[No slacking dual] \label{noslackingdual} For each $d\ge 1$ and $q\in(2,\infty)$ there exist $\delta_0,c,C_0<\infty$ with the following property. Let $\delta\in(0,\delta_0]$ and let $f=g+h$ where $f,g,h\in L^p(\R^d)$ and $g,h$ are disjointly supported on $A,B$ respectively. Suppose that 
\[ \|\widehat{f}\|_{q,*}\ge (1-\delta){\bf{B}}_{q,d}\|f\|_p.\]
and that  
\[ \|h\|_p\ge C_0\delta^{1/p}\|f\|_p .\]
Then 
\[ \|\widehat{h}\|_{q,\infty}\ge c\delta\|f\|_p.\]
\end{lemma}

\begin{proof} Using the hypothesis that $f$ is near-extremizing,  
\begin{align*} 
(1-\delta)\|f\|_p{\bf{B}}_{q,d}\le \|\widehat{f}\|_{q,*}&\le \|\widehat{g}\|_{q,*}+\|\widehat{h}\|_{q,*} \\ 
&\le {\bf{B}}_{q,d}\|g\|_p+\|\widehat{h}\|_{q,*} \\
&\le {\bf{B}}_{q,d}(\|f\|_p^p-\|h\|^p_p)^{1/p}+\|\widehat{h}\|_{q,*} \\
&\le {\bf{B}}_{q,d}(\|f\|_p^p-C_0^p\delta\|f\|^p_p)^{1/p}+\|\widehat{h}\|_{q,*} . 
\end{align*}
Rearranging the above inequality gives 
\begin{align*}
\left((1-\delta)-(1-C_0^p\delta)^{1/p}\right){\bf{B}}_{q,d}\|f\|_p&\le \|\widehat{h}\|_{q,*} .   
\end{align*}
Finally, we can arrange that $|C_0^p\delta|<1$, so 
\begin{align*}
1-\delta-(1-C_0^p\delta)^{1/p} &= -\delta+\frac{1}{p}C_0^p\delta+O(\delta^2) \\
&= (C_0^p/p-1)\delta+O(\delta^2).
\end{align*}

If $C_0^p/p-2>0$ and $\delta$ is small enough, we have the result.  

\end{proof}

\subsection{Cooperation }

\begin{lemma}\label{coop} Let $p\in[1,2)$ and $q\in[2,\infty)$. There exist $c,C\in\R^+$ with the following property. Let $0\not=f\in L^p$ satisfy $|f|\le 1_E$ and $\|\widehat{f}\|_q\ge(1-\delta){\bf{B}}_{q,d}|E|^{1/p}$. Suppose that $f=f^\sharp+f^\flat$ where $\text{supp }f^\sharp=A$ and $\text{supp }f^\flat=B$ satisfy
\[ A\cup B= E, \quad A\cap B=\emptyset, \]
\[\text{and}\quad \min(|A|,|B|)\ge \eta^p|E|.\]
Then 
\[ \|\widehat{f^\sharp}\cdot\widehat{f^\flat}\|_{q/2}\ge (c\eta^p-C\delta)|E|^{2/p}.\]
\end{lemma}
\begin{proof} 
\begin{align*}
    \|\widehat{f}\|_q^q&\le \int(|\widehat{f^\sharp}|^2+|\widehat{f^\flat}|^2)|\widehat{f}|^{q-2}+2\int|\widehat{f^\sharp}\cdot\widehat{f^\flat}||\widehat{f}|^{q-2} \\
    &\le (\||\widehat{f^\sharp}|^2\|_{q/2}+\||\widehat{f^\sharp}|^2\|_{q/2})\||\widehat{f}|^{q-2}\|_{q/(q-2)}+2\|\widehat{f^\sharp}\cdot\widehat{f^\flat}\|_{q/2}\||\widehat{f}|^{q-2}\|_{q/(q-2)} \\
    &=(\|\widehat{f^\sharp}\|_{q}^2+\|\widehat{f^\sharp}\|_{q}^2)\|\widehat{f}\|_{q}^{q-2}+2\|\widehat{f^\sharp}\cdot\widehat{f^\flat}\|_{q/2}\|\widehat{f}\|_{q}^{q-2} \\
    &\le (|A|^{2/p}+|B|^{2/p}){\bf{B}}_{q,d}^{q}|E|^{(q-2)/p}+2\|\widehat{f^\sharp}\cdot\widehat{f^\flat}\|_{q/2}{\bf{B}}_{q,d}^{q-2}|E|^{(q-2)/p}.
\end{align*}
Rearranging gives
\begin{align*}  
\|\widehat{f^\sharp}\cdot\widehat{f^\flat}\|_{q/2}&\ge (2{\bf{B}}_{q,d}^{q-2}|E|^{(q-2)/p})^{-1}\left(\|\widehat{f}\|_q^q-(|A|^{2/p}+|B|^{2/p}){\bf{B}}_{q,d}^{q}|E|^{(q-2)/p}\right)\\
&\ge (2{\bf{B}}_{q,d}^{q-2}|E|^{(q-2)/p})^{-1}\left((1-\delta)^q{\bf{B}}_{q,d}^q|E|^{q/p}-(|A|^{2/p}+|B|^{2/p}){\bf{B}}_{q,d}^{q}|E|^{(q-2)/p}\right)\\
&\ge 2^{-1}{\bf{B}}_{q,d}^2\left((1-\delta)^q|E|^{2/p}-|A|^{2/p}-|B|^{2/p}\right).
\end{align*} 
Note that since $p<2$,
\[ (|A|^{2/p}+|B|^{2/p})^{p/2}\le |A|+|B|\le |E|\]
with strict inequality unless $|A|$ or $|B|$ is 0. Without loss of generality, suppose that $|E|=1$. 

We want to show there exists $c\in\R^+$ such that for $\eta$ small enough and $\eta^p\le \min(|A|,|B|)$,
\[ \frac{|A|^{2/p}+|B|^{2/p}}{(|A|+|B|)^{2/p}}=|A|^{2/p}+|B|^{2/p}\le 1-c\eta^p .\]
By assumption, $|A|,|B|\in[\eta^p,1-\eta^p]$, so $|A|^{2/p}+|B|^{2/p}\le (\eta^p)^{2/p}+(1-\eta^p)^{2/p}$. For all $\eta>0$ sufficiently small, there exists $c>0$ so that $(\eta^p)^{2/p}+(1-\eta^p)^{2/p}\le 1-c\eta^p$. 

Finally, using $|A|+|B|=|E|=1$,
\begin{align*} 
\|\widehat{f^\sharp}\cdot\widehat{f^\flat}\|_{q/2}&\ge 2^{-1}{\bf{B}}_{q,d}^2\left((1-\delta)^q-|A|^{2/p}-|B|^{2/p}\right) \\
&\ge 2^{-1}{\bf{B}}_{q,d}^2\left((1-\delta)^q-(1-c\eta^p)\right)  \\
&\ge c\eta^p-C\delta .
\end{align*} 

\end{proof}

\begin{lemma}\label{coopdual} For each $d\ge 1$ and $q\in(2,\infty)$ there exist $\delta_0,c,C_0<\infty$ with the following property. Let $\delta\in(0,\delta_0]$ and let $f=g+h$ where $f,g,h\in L^p(\R^d)$ and $g,h$ are disjointly supported. Let $\eta^p\ge\delta$. Suppose that the following inequalities hold. 
\[ \|\widehat{f}\|_{q,*}\ge (1-\delta){\bf{B}}_{q,d}\|f\|_p,\]
\[ \min(\|g\|_p,\|h\|_p)\ge C_0\eta\|f\|_p . \]

Then 
\[ \||\widehat{g}|^{1/2}|\widehat{h}|^{1/2}\|_{q,*}\ge c\delta\|f\|_p{\bf{B}}_{q,d}.\]

\end{lemma}

\begin{proof} 
Take $E\subset\R^d$ with $|E|\in(0,\infty)$ satisfying
\begin{align*}
    |E|^{-1/p}\int_E|\widehat{f}|&\ge (1-2\delta){\bf{B}}_{q,d}\|f\|_p.
\end{align*}
By replacing $E$ with $E\cap \{\widehat{f}\not=0\}$, we can assume that $\widehat{f}$ is nonzero on $E$. For $\lambda>0$ a large constant to be chosen later, define $E_{\lambda,g}= \{x\in E:|\widehat{g}|>\lambda|\widehat{h}|\}$ and $E_{\lambda,h}=\{x\in E:|\widehat{h}|>\lambda|\widehat{g}|\}$. Note that

\begin{align*}
\int_E|\widehat{f}| &=     \int_{E_{\lambda,g}}|\widehat{f}| +    \int_{E_{\lambda,h}}|\widehat{f}| +    \int_{E\setminus(E_{\lambda,g}\cup E_{\lambda,h})}|\widehat{f}|   \\
&\le      (1+1/\lambda)\int_{E_{\lambda,g}}|\widehat{g}| +    (1+1/\lambda)\int_{E_{\lambda,h}}|\widehat{h}| +   \int_{E\setminus(E_{\lambda,g}\cup E_{\lambda,h})}(|\widehat{g}|+|\widehat{h}|)   \\
&\le      (1+1/\lambda)\int_{E_{\lambda,g}}|\widehat{g}| +    (1+1/\lambda)\int_{E_{\lambda,h}}|\widehat{h}| +   \int_{E\setminus(E_{\lambda,g}\cup E_{\lambda,h})}(|\widehat{g}|^{1/2}\lambda^{1/2}|\widehat{h}|^{1/2}+\lambda^{1/2}|\widehat{g}|^{1/2}|\widehat{h}|^{1/2})\\
&\le      (1+1/\lambda)(|E_{\lambda,g}|^{1/p}\|\widehat{g}\|_{q,*}+|E_{\lambda,h}|^{1/p}\|\widehat{h}\|_{q,*}) +   2\lambda^{1/2}\int_{E\setminus(E_{\lambda,g}\cup E_{\lambda,h})}|\widehat{g}|^{1/2}|\widehat{h}|^{1/2}.
\end{align*}

Using our main dual inequality, we have 
\begin{align*}
|E_{\lambda,g}|^{1/p}\|\widehat{g}\|_{q,*}+|E_{\lambda,h}|^{1/p}\|\widehat{h}\|_{q,*}&\le \left(|E_{\lambda,g}|^{1/p}\|g\|_p+|E_{\lambda,h}|^{1/p}\|h\|_p\right){\bf{B}}_{q,d}  
\end{align*} 
and by H\"{o}lder's inequality, 
\[|E_{\lambda,g}|^{1/p}\|g\|_p+|E_{\lambda,h}|^{1/p}\|h\|_p  \le (|E_{\lambda,g}|^{p/p}+|E_{\lambda,h}|^{p/p})^{1/p}(\|g\|_p^q+\|h\|_p^q)^{1/q}  \le  |E|^{1/p}(\|g\|_p^q+\|h\|_p^q)^{1/q}.\]
Also 
\[   \|g\|_p^q+\|h\|_p^q\le \max(\|g\|_p^{q-p},\|h\|_p^{q-p})(\|g\|_p^p+\|h\|_p^p)=\max(\|g\|_p^{q-p},\|h\|_p^{q-p})\|f\|_p^p.   \]
Now we use the hypothesis that $\min(\|g\|_p,\|h\|_p)\ge C_0\eta \|f\|_p$ to say 
\[ \max(\|g\|^p_p,\|h\|^p_p)=\|f\|_p^p-\min(\|g\|_p,\|h\|_p)\le \|f\|_p^p(1-C_0^p\eta^p). \]

In summary, 
\begin{align*}
|E_{\lambda,g}|^{1/p}\|\widehat{g}\|_{q,*}+|E_{\lambda,h}|^{1/p}\|\widehat{h}\|_{q,*}&\le |E|^{1/p}\left(\|f\|_p^p)^{1/q}\|f\|_p^{(q-p)/q}(1-C_0^p\eta^p)^{(q-p)/q}\right){\bf{B}}_{q,d}      \\
&=  |E|^{1/p}\|f\|_p(1-C_0^p\eta^p)^{(q-p)/q}{\bf{B}}_{q,d}. 
\end{align*}

Putting everything together, we have 
\begin{align*}
    (1-2\delta){\bf{B}}_{q,d}\|f\|_p\le(1+1/\lambda) \|f\|_p(1-&C_0^p\eta^p)^{(q-p)/q}{\bf{B}}_{q,d}+\lambda^{1/2} |E|^{-1/p}\int_{E\setminus (E_{\lambda,h}\cup E_{\lambda,h})}|\widehat{g}\widehat{h}|^{1/2} \\
    (1-2\delta-(1+1/\lambda)(1-(1-p/q)C_0^p\eta^p&+O(\eta^{2p})){\bf{B}}_{q,d}\|f\|_p\le\lambda^{1/2} |E|^{-1/p}\int_{E\setminus (E_{\lambda,h}\cup E_{\lambda,h})}|\widehat{g}\widehat{h}|^{1/2} \\
    (-2\delta-1/\lambda+(1+1/\lambda)(1-p/q)C_0^p\eta^p&+O(\eta^{2p})){\bf{B}}_{q,d}\|f\|_p\le\lambda^{1/2} |E|^{-1/p}\int_{E}|\widehat{g}\widehat{h}|^{1/2} .
\end{align*}

The desired inequality follows from choosing $\lambda=\delta^{-1}$, $\eta^p\ge\delta$ and $C_0$ large enough.

\end{proof}

\begin{section}{Multiprogression structure of quasi-extremizers}

In this section, we relate quasi-extremizers for (\ref{eqn:main}) to quasi-extremizers for Young's convolution inequality. Then we exploit the connection between Young's convolution inequality and principles of additive combinatorics which imply that quasi-extremizing functions for Young's inequality have significant support on sets with arithmetic structure. We use the following definition and notation for multiprogressions. 

\begin{definition}
A discrete multiprogression ${\mathbf P}$ in $\R^d$ of rank $r$ is a function
\[{\bf{P}}: \prod_{i=1}^r\{0,1,\dots,N_i-1\} \to \R^d\] of the form
\[ {\mathbf P}(n_1,\dots,n_r)=\big\{a + \sum_{i=1}^r n_i v_i: 0\le n_i<N_i\big\}, \]
for some $a\in\R^d$, some $v_j\in\R^d$, 
and some positive integers $N_1,\dots,N_r$. A continuum multiprogression $P$ in $\R^d$ of rank $r$ is a function 
\[P: \prod_{i=1}^r \{0,1,\dots,N_i-1\}\times[0,1]^d \to \R^d\] of the form
\[(n_1,\dots,n_d;y)\mapsto a+\sum_i n_i v_i + sy\] where $a,v_i\in\R^d$ and $s\in\R^+$. 
The {\it size} of $P$ is defined to be \[\sigma(P)=s^d \prod_i N_i.\]
$P$ is said to be {\it proper} if this mapping is injective.
\end{definition}

We will often identify a multiprogression with its range, 
and will refer to multiprogressions as if they were sets rather than functions. If $P$ is proper then the Lebesgue measure of its range equals its size. For a discussion of properties of multiprogressions, see \textsection{5} of \cite{c1}.

\begin{lemma}[Quasi-extremizers for Young's inequality]\label{quasi-y} Let $r\in (1,\infty)$ and suppose that the exponent $t$ defined by $1+t^{-1}=2r^{-1}$ also belongs to $(1,\infty)$. For each $\delta>0$, there exist $c_\delta,C_\delta\in(0,\infty)$ such that for any $|f|\le 1_E$ with $0<|E|<\infty$ and $|E|^{2/r}\delta \le \|f*f\|_t$, there exist a disjoint, measurable partition $E=A\cup B$ and a proper continuum multiprogression $P$ such that 
\[ A\subset P\]
\[ |P|\le C_\delta |A|\]
\[ \text{rank }(P)\le C_\delta  \]
\[ \|f-1_Af\|_{r}\le (1-c_\delta)\|f\|_{r}.\]
\end{lemma} 
\begin{proof} This lemma follows from the proof of Lemma 6.1 in \cite{c1} where we specialize to the case $f_1=f_2$ and use the relation $|E|^{2/r}\ge  \|f\|_r^2$. 

\end{proof}

\begin{lemma}\label{3lineslemma} Let $d\ge1$ and $p\in(1,2)$. Let $\eta>0$. Suppose that $E$ is a  measurable set and $f$ is a nonzero function satisfying $|f|\le 1_E\in L^p(\R^d)$ and $|E|^{1/p}\eta\le \|\widehat{f}\|_{p'}$. If $p\le 4/3$, $|E|^{2/p}\eta^2\le \||f|*|f|\|_t$ for $t^{-1}=2p^{-1}-1$. If $4/3<p$, then there exists $\g=\g(p)\in\R^+$ such that $|E|^{3/2}\eta^\g\le\||f|^{4/3}*|f|^{4/3}\|_2$. 
\end{lemma} 

\begin{proof} First suppose that $p\le \frac{4}{3}$. Then applying Plancherel's theorem and the Hausdorff-Young inequality, we have 
\[ |E|^{1/p}\eta\le \|\widehat{f}\|_{p'}=\|\widehat{f*f}\|^{1/2}_{p'/2}\le \|f*f\|^{1/2}_{(p'/2)'}\le \||f|*|f|\|^{1/2}_{t},\]
where $t=\frac{p'/2}{p'/2-1}=\frac{p/(p-1)}{p/(p-1)-2}=(2p^{-1}-1)^{-1}$.

Write $f(x)=g(x)e^{i\p(x)}$ where $\p(x)$ is real-valued and $g\ge0$. Note that for $\text{Re}{z}>0$, we can define $f_z:=g^ze^{i\p}\in L^{p/{\text{Re}{z}}}$.

Assume that $\frac{4}{3}<p$. Since $\frac{p}{2}<1<\frac{3p}{4}$, there exists $\theta\in(0,1)$ such that $1=(1-\theta)p2^{-1}+\theta 3p4^{-1}$. By the Three Lines Lemma proof of the Riesz-Thorin theorem,

\[ \|\widehat{f}\|_{p'}\le \sup_{\text{Re}z=p/2}\|\widehat{f_z}\|^{1-\theta}_{2}\sup_{\text{Re}z=3p/4}\|\widehat{f_z}\|_{(4/3)'}^\theta =\|f\|^{(1-\theta)p2^{-1}}_{p}\sup_{\text{Re}z=3p/4}\|\widehat{f_z}\|_{(4/3)'}^\theta.\]
Combining this with the quasi-extremal hypothesis for $f$ gives
\begin{align*}  
|E|^{1/p}\eta&\le \|f\|_p^{(1-\theta)p2^{-1}}\sup_{\text{Re }z=4/3}\|\widehat{f_z}\|^\theta_{(4/3)'} \\
&\le |E|^{(1-\theta)2^{-1}}\sup_{\text{Re }z=4/3}\|\widehat{f_z*f_z}\|^{\theta/2}_{2} \\
&= |E|^{(1-\theta)2^{-1}}\sup_{\text{Re }z=4/3}\|f_z*f_z\|^{\theta/2}_{2} \\
&\le |E|^{(1-\theta)2^{-1}}\||f|^{4/3}*|f|^{4/3}\|^{\theta/2}_{2}.
\end{align*} 
Rearranging, we can write 
\[ |E|^{3\theta/4}\eta\le \||f|^{4/3}*|f|^{4/3}\|^{\theta/2}_2 \]
so $|E|^{3/2}\eta^\gamma\le \||f|^{4/3}*|f|^{4/3}\|_2$ for some $\g>0$. 

\end{proof}

\begin{proposition}[Structure of quasi-extremizers]\label{quasistructure} Let $d\ge 1$, let $\Lambda\subset(1,2)$ be a compact set, and let $\eta>0$. There exist $C_\eta,c_\eta\in \R^+$ with the following property for all $p\in\Lambda$. Suppose that $0\not=f\in L^p(\R^d)$, $|f|\le 1_{E}$ with $|E|<\infty$, and $\|\widehat{f}\|_{q}\ge \eta|E|^{1/p}$. Then there exists a multiprogression $P$ and a disjoint, measurable partition $E=A\cup B$ such that 
\[ A\subset P\]
\[ |P|\le C_\eta|A|\]
\[  \text{rank }P\le C_\eta  \]
\[ \|f-1_Af\|_p\le (1-c_\eta)\|f\|_p .\] 
\end{proposition} 
\begin{proof} Combine Lemmas \ref{3lineslemma} and \ref{quasi-y}.
\end{proof}

\begin{lemma}\label{quasistructuredual} Let $d\ge 1$, let $\Lambda\subset (2,\infty)$ be compact, and let $\eta\in(0,1]$. There exist $C_\eta,c_\eta>0$ with the following property for all $q\in\Lambda$. Suppose that $0\not=f\in L^{q'}(\R^d)$ satisfies $\|\widehat{f}\|_{q,*}\ge \eta\|f\|_p$. Then there exist a proper continuum multiprogression $P$ and a disjointly supported Lebesgue measurable decomposition $f=g+h$ such that 
\begin{align*}
    g\prec P,\\
\|g\|_p\ge c_\eta\|f\|_p\\
\|g\|_\infty |P|^{1/p}\le C_\eta\|f\|_p \\
\text{rank }P\le C_\eta.
\end{align*}

\end{lemma}
\begin{proof}
This follows from Proposition 6.4 in \cite{c1} since $\|f\|_p\eta\le \|\widehat{f}\|_{q,*}\le \|\widehat{f}\|_q$. 
\end{proof}
\end{section}

\begin{section}{Multiprogression structure of near-extremizers}

The following is a restatement of Lemma 5.5 of \cite{c1}, included here for the reader's convenience.  
\begin{lemma}[Compatibility of nonnegligibly interacting multiprogressions]\label{multicompat} Let $d\ge 1$. Let $\Lambda$ be a compact subset of $(1,2)$. Let $\lambda>0$ and $R<\infty$. There exists $C<\infty$, depending only $\lambda,R,d,\Lambda$, with the following property. Let $p\in\Lambda$. Let $P,Q\subset{\R^d}$ be nonempty proper continuum multiprogressions of ranks $\le R$. Let $\p\prec P$ and $\psi\prec Q$ be functions that satisfy $\|\p\|_\infty|P|^{1/p}\le 1$ and $\|\psi\|_\infty|Q|^{1/p}\le 1$. If 
\[ \|\widehat{\p}\widehat{\psi}\|_{q/2}\ge\lambda\]
then
\[ \max(|P|,|Q|)\le C\min(|P|,|Q|)\]
\[ |P+Q|\le C\min(|P|,|Q|). \]

\end{lemma}

\begin{lemma}\label{simplenearextstructure}Let $d\ge 1$, and let $\Lambda\subset(1,2)$ be a compact set. For any $\epsilon>0$ there exist $\delta>0$, $N_\epsilon<\infty$, and $C_\epsilon<\infty$ with the following property for all $p\in\Lambda$. Let $|E|<\infty$ and $|f|\le 1_E$ be such that  $\|\widehat{f}\|_q\ge (1-\delta){\bf{B}}_{q,d}|E|^{1/p}$. Then there exist a measurable decomposition $f=g+h$, where $g=g1_A$, $h=h1_B$, and $A\cap B=\emptyset$, and continuum multiprogressions $\{P_i:1\le i\le N_\epsilon\}$ such that 
\[ |B|\le \epsilon |E| \]
\[ \sum_i|P_i|\le C_\epsilon |E| \]
\[ A\subset \overset{N_\epsilon}{\underset{i=1}{\cup}} P_i\]
\[ \emph{rank}\,P_i\le C_\epsilon\]
\[\|g\|_p\ge c_\delta \|f\|_p.\]
\end{lemma}
\begin{proof} We define an iterative process following the proof of Theorem 7.1 from \cite{c1}. Setting $\eta_\delta=1-\delta$, we may apply Proposition \ref{quasistructure} to obtain a disjoint decomposition $E=A_1\cup B_1$ with a multiprogression $P_1$ satisfying 
\[ |P_1|\le C_{\eta_\delta}|A_1|,\quad\text{rank }P_1\le C_{\eta_\delta},\quad \|1_{A_1}f\|_p\ge c_{\eta_\delta}\|f\|_p.\]
Suppose that $|B_1|>\epsilon|E|$ (the case $|B_1|\le \epsilon |E|$ will be analyzed below). By Lemma \ref{noslacking} with $\delta<\epsilon/C_0$, 
\[ \|\widehat{1_{B_1}f}\|_q\ge \frac{c}{C_0}\epsilon |E|^{1/p},\]
where $c, C_0$ are as in the lemma. Define $\eta_\epsilon =\frac{c}{C_0}\epsilon$. Then we apply Proposition \ref{quasistructure} to $1_{B_1}f$ to obtain a disjoint decomposition $B_1=A_2\cup B_2$ with the corresponding conclusions. 

For the $k$-th step in the process, we halt if $|B_{k-1}|\le \epsilon|E|$. If $|B_{k-1}|>\epsilon |E|$, then by Lemma \ref{noslacking}, we have $\|\widehat{1_{B_{k-1}}f}\|_q\ge \eta_\epsilon |E|^{1/p}$. Then applying Proposition \ref{quasistructure}, we get $B_{k-1}=A_k\cup B_k$ with the conclusions of the proposition.

We note that this process terminates after finitely many steps since all of the $B_i$ are disjoint and after $m$ steps, $|E|\ge |B_1|+\cdots+|B_m|>m\epsilon |E|$. Thus we may suppose we have obtained a disjoint decomposition 
\[ E=A_1\cup\cdots\cup A_n\cup B_n\]
where $|B_i|>\epsilon|E|$ for $1\le i<n$ and $|B_n|\le \epsilon|E|$. We also have multiprogressions $P_i$ satisfying $|P_1|\le C_{\eta_\delta}|A_1|$, $\text{rank }P_1\le C_{\eta_\delta}$ and for $1<i\le n$, $|P_i|\le C_{\eta_\epsilon}|A_i|$, $\text{rank }P_i\le C_{\eta_\epsilon}$. Thus 
\[ \sum_i|P_i|\le C_\epsilon |E|,\]
$A:=\underset{i}{\cup} A_i\subset \underset{i}{\cup} P_i$, $\text{rank }P_i\le C_\epsilon$, and 
\[ \|1_Af\|_p\ge \|1_{A_1}f\|_p\ge c_\delta \|f\|_p,\]
as desired.
\end{proof} 

\begin{lemma}[More structured decomposition]\label{nearextstructure} Let $d\ge 1$, and let $\Lambda\subset(1,2)$ be a compact set. For any $\epsilon>0$ there exist $\delta>0$, $N_\epsilon<\infty$, and $C_\epsilon<\infty$ with the following property for all $p\in\Lambda$. Let $|E|<\infty$ and $|f|\le 1_E$ be such that  $\|\widehat{f}\|_q\ge (1-\delta){\bf{B}}_{q,d}|E|^{1/p}$. Then there exist a measurable decomposition $f=g+h$, where $g=g1_A$, $h=h1_B$, and $A\cap B=\emptyset$, and a continuum multiprogression $P$ such that 
\[ |B|\le \epsilon |E| \]
\[ |P|\le C_\epsilon |E| \]
\[ A\subset P\]
\[ \emph{rank}\,P\le C_\epsilon.\]
\end{lemma}

\begin{proof} First we define $E_\lambda=\{x\in E:|f(x)|\le\lambda\}$. Note that by the Hausdorff-Young inequality,
\[ \|\widehat{1_{E_{\lambda}}f}\|_q\le \|1_{E_\lambda}f\|_p\le \lambda|E|^{1/p} .\]
Assume that $|E_\lambda|>\epsilon|E|$. Then by Lemma \ref{noslacking},

\[ \|\widehat{1_{E_\lambda}f}\|_q\ge \frac{c_0\epsilon}{C} |E|^{1/p}:=\eta_\epsilon|E|^{1/p}. \]
Thus if we take $\lambda=\eta_\epsilon \epsilon $, we are guaranteed that $|E_\lambda|<\epsilon|E|$. Now without loss of generality, assume that $|f|\ge \eta_\epsilon\epsilon$ on $E$. 

We define an iterative process with an outer and an inner loop. For the step 1 of the outer loop, letting $\eta_\delta=1-\delta$, apply Proposition \ref{quasistructure} to get $E=A_1\cup B_1$ where $A_1$ is contained in a multiprogression $P_1$ satisfying the conclusions of the proposition. At step $N$ of the outer loop, we have a measurable decomposition 
\[ f=G_N+H_N\]
where $H_N=1_{B_N}H_N$ and $G_N=1_{A_N}G_N$, where $A_N\cap B_N=\emptyset$ and $A_N$ is contained in a multiprogression $P_N$ with $|P_N|\le C_\epsilon|E|$, $\text{rank }P_N\le C_\epsilon$, and $\|G_N\|_p\ge c_\delta \|f\|_p$. If $|B_N|<\epsilon |E|$, then we halt.  Otherwise, initiate step $(N,1)$ of the inner loop. Since $|B_N|\ge \epsilon|E|$, by Lemma \ref{noslacking}, $\|\widehat{1_{B_N}f}\|_q\ge \eta_\epsilon|E|^{1/p}$. Thus we can decompose $B_N$ into $S_{N,1}$ (contained in a multiprogression) and $R_{N,1}$ using Proposition \ref{quasistructure}. The halting criterion for the $(N,j)$th step is $|R_{N,j}|\le \frac{1}{2}\epsilon|E|$ or $\|\widehat{G_N}\widehat{1_{S_{N,j}}f}\|_{q/2}\ge\rho|E|^{2/p}$. If neither is satisfied in step $(N,j)$, then $|R_{N,j}|>\frac{1}{2}\epsilon|E|$, so repeat the argument described for step $(N,1)$ replacing $B_N$ by $R_{N,j}$. After $k$ iterations of the inner loop, we note that
\[|B_N|\ge  |R_{N,1}|+\cdots+|R_{N,k}|\ge k\epsilon|E|,\]
so the inner loop terminates in a maximum of $M_\epsilon$ steps. 

Suppose that the inner loop terminates at step $k$ because $|R_{N,k}|\le \frac{1}{2}\epsilon|E|$ but $\|\widehat{G_N}\widehat{1_{S_{N,k}}f}\|_{q/2}< \rho|E|^{2/p}$. Then $\|\widehat{G_N}\widehat{1_{S_{N,j}}f}\|_{q/2}< \rho|E|^{2/p}$ for $1\le j\le k$. Define $h=\sum\limits_{j=1}^k1_{S_{N,j}}f$. Note that
\begin{equation}\label{structurecontradiction}
\|\widehat{G_N}\widehat{h}\|_{q/2}\le \sum_{j=1}^k\|\widehat{G_N}\widehat{1_{S_{N,j}}f}\|_{q/2}< M_\epsilon \rho|E|^{2/p}.
\end{equation}    

However, $|\text{supp }h|=\sum\limits_{j=1}^k|S_{N,k}|\ge |B_N|-|R_{N,k}|\ge\epsilon|E|-\frac{1}{2}\epsilon|E|=\frac{\epsilon}{2}|E|$ and 
\[ |\text{supp }G_N|=|A_N|\ge \|G_N\|_p^p\ge c_{\delta}\|f\|_p^p\ge c_\delta \eta_\epsilon \epsilon |E|^{1/p}\]
where we used the assumption that $|f|\ge \eta_\epsilon \epsilon$ discussed at the beginning of the proof. Finally note that $\|\widehat{G_N+h}\|
_q\ge\|\widehat{f}\|_q-\|\widehat{1_{R_{N,k}}f}\|q\ge (1-\epsilon-\epsilon^{1/p}){\bf{B}}_{q,d}|E|^{1/p}$. Thus, choosing $\delta$ and $\rho$ small enough, (\ref{structurecontradiction}) contradicts Lemma \ref{coop}. 

Thus the halting criterion for the inner loop yields a function $1_{S_{N,k}}f$ such that  
\begin{equation} \label{terminates1}
\|\widehat{G_N}\widehat{1_{S_{N,k}}f}\|_{q/2}\ge \rho|E|^{2/p}.
\end{equation} 
The function $1_{S_{N,k}}f$ also satisfies 
\begin{equation} \label{terminates2} \|1_{S_{N,k}}f\|_{p}\ge c_\epsilon \|1_{R_{N,k-1}}f\|_p\ge c_\epsilon \eta_\epsilon \epsilon |R_{N,k-1}|^{1/p}\ge c_\epsilon\eta_\epsilon \epsilon^{1+1/p}|E|^{1/p}.
\end{equation} 

If $Q_N$ is the multiprogression associated to $S_{N,k}$, then Lemma \ref{multicompat} (taking $\p=\frac{1}{C_\epsilon |E|}1_{A_N}f$ and $\psi=\frac{1}{C_\epsilon|E|}1_{S_{N,k}}f$, which satisfies the hypotheses for small enough $\rho$) implies that $|P_N+Q_N|\le C'_\epsilon \min(|P_N|,|Q_N|)$. Thus there exists a continuum multiprogression $P_{N+1}$ of rank $\le C_\epsilon$ containing $P_N$ and $Q_N$ and satisfying $|P_{N+1}|\le C_\epsilon|E|$. 

Set $G_{N+1}=G_N+1_{S_{N,k}}f$. Then $H_{N+1}:=f-G_{N+1}$ has support called $B_{N+1}$. If $|B_{N+1}|\le\epsilon|E|$, then we're done. If not, proceed to outer loop step $N+2$. Note that for each outer loop step, we have 
\[ \|G_{N+1}\|_{p}^p\ge \|G_{N}\|_p^p+\|1_{S_{N,k}}f\|_p^p\ge \|G_N\|_p^p+c_\epsilon\eta_\epsilon \epsilon^{p+1}|E|. \]

Thus the outer loop terminates in at most $N_\epsilon$ steps. Note that since the ranks of $P_N$ and $Q_N$ at most add at each step of the outer loop, the rank of the ultimate multiprogression is controlled by $M_\epsilon>0$.

\end{proof}

\begin{lemma}\label{nearextstructuredual} Let $d\ge 1$, and let $\Lambda\subset(1,2)$ be a compact set. For any $\epsilon>0$ there exist $\delta>0$, $N_\epsilon<\infty$, and $C_\epsilon<\infty$ with the following property for all $p\in\Lambda$. Let $|E|<\infty$ and $|f|\le 1_E$ be such that  $\|\widehat{f}\|_{q,*}\ge (1-\delta){\bf{B}}_{q,d}\|f\|_p$. Then there exists a measurable decomposition $f=g+h$ where $g=1_Ag$, $h=1_Bh$, $A\cap B=\emptyset$, and there is a continuum multiprogression $P$ such that 
\[ \|h\|_p\le \epsilon \|f\|_p \]
\[ \|g\|_\infty|P|^{1/p}\le C_\epsilon \|f\|_p \]
\[ A\subset P\]
\[ \emph{rank }P\le C_\epsilon.\]
\end{lemma}

\begin{proof} Using the hypothesis $\|\widehat{f}\|_{q,*}\ge (1-\delta){\bf{B}}_{q,d}\|f\|_p$, by Lemma \ref{quasistructuredual} there exists a disjoint decomposition $f=g_1+h_1$ where $g_1$ is supported on a multiprogression $P_1$ with $\text{rank }P_1\le C_\delta$, $\|g_1\|_p\ge c_\delta \|f\|_p$, $\|g_1\|_{\infty}|P_1|^{1/p}\le C_\delta \|f\|_p$. If $\|h_1\|_p<\epsilon \|f\|_p$, then we halt.

To further refine the decomposition in the case that $\|h_1\|_p\ge \epsilon\|f\|_p$, define an iterative process with input $(g_1,h_1)$ and output $(g_2,h_2)$ where $f=g_2+h_2$ and $g_2,h_2$ satisfy certain properties below. Apply Lemma \ref{noslackingdual} to conclude that $\|\widehat{h_1}\|_{q,*}\ge \eta_\epsilon\|f\|_p$ for $\eta_\epsilon>0$. Then apply Lemma \ref{quasistructuredual} to get $h_1=u_1+v_1$ where $u_1$  is supported on a multiprogression $Q_1$, $\text{rank }Q_1\le C_\epsilon$, $\|u_1\|_\infty|Q_1|^{1/p}\le C_\epsilon \|f\|_p$, and $\|u_1\|_p\ge c_\epsilon \|h_1\|_p\ge c_\epsilon \epsilon\|f\|_p$. 

Choose $\delta$ suffciently small to ensure that $c_\delta\ge\epsilon c_\epsilon$. Since $\min(\|g_1\|_p,\|u_1\|_p)\ge\epsilon c_{\epsilon}\|f\|_p$, by  Lemma \ref{coopdual}, $\||\widehat{g_1}|^{1/2}|\widehat{u_1}|^{1/2}\|_{q,*}\ge \rho(\epsilon)\|f\|_p{\bf{B}}_{q,d}$ for $\rho(\epsilon)>0$. But then Lemma \ref{multicompat} (taking $\p=\frac{1}{C_\delta \|f\|_p}g_1$ and $\psi=\frac{1}{C_\epsilon\|f\|_p}u_1$) implies that $\max(|P_1|,|Q_1|)\le C'_\epsilon\min(|P_1|,|Q_1|)$ and $|P_1+Q_1|\le C'_\epsilon \min(|P_1|,|Q_1|)$. Thus there exists a continuum multiprogression $P_{2}$ of rank $\le C_{\epsilon,\delta}$ containing $P_1$ and $Q_1$ and satisfying $|P_{2}|\le C_{\epsilon}$. Define $g_2:= g_1+u_1$ and $h_2:=v_1$. 

If $\|h_2\|_p<\epsilon\|f\|_p$, then halt. If $\|h_2\|_p\ge \epsilon\|f\|_p$, repeat the process described above with input $(g_2,h_2)$. 

After $n$ steps of this iteration, we have a decomposition $f=g_n+h_n$ and a multiprogression $P_n$ of controlled size and rank containing the support of $g_n$ and satisfying $\|g_n\|_\infty|P_n|^{1/p}\le C_{\epsilon}\|f\|_p$, and
\[ \|g_n\|_p^p= \|g_1\|_p^p+\|u_1\|_p^p+\cdots+\|u_{n-1}\|_p^p\ge  (c_\delta^p+(n-1)c_\epsilon^p \epsilon^p )\|f\|_p^p. \]

Thus the loop terminates in at most $n_\epsilon$ steps. Note that since the ranks of $P_n$ and $Q_n$ at most add at each step of the process, the rank of the ultimate multiprogression is controlled by a constant depending on $\epsilon$. Also, $|P_n|\le (C_{\epsilon}')^{n-1}(\min(|P_1|,|Q_1|,\ldots,|Q_{n-1}|)$. 

Finally we note that 
\begin{align*} 
\|g_n\|_\infty |P_n|^{1/p}&\le (C_\epsilon')^{(n-1)} (\|g_1\|_\infty|P_1|^{1/p}+\|u_1\|_\infty|Q_1|^{1/p}+\cdots+\|u_{n-1}\|_\infty|Q_{n-1}|^{1/p}) \\
&\le (C_\epsilon')^{n}(n-1)\|f\|_p.
\end{align*}

\end{proof}

\end{section}

\section{Exploitation of $\Z^\kappa\times\R^d$\label{product}}
\begin{subsection}{Analysis of the discrete Hausdorff-Young inequality}

Let $\T$ denote the quotient group $\R/\Z$. Extend the previous notation and define the Fourier transform $\widehat{\cdot}:\Z^\kappa\times\R^d\to \T^\kappa\times\R^d$ by 
\[ \widehat{f}(\theta,\xi)=\int_{\R^d}\sum_{n\in\Z^\kappa}e^{-2\pi ix\cdot\xi}e^{-2\pi i n\cdot\theta}f(n,x)dx \]
where $\theta\in\T^d$. This can be decomposed as $\F\circ\tilde{\F}$ where 
\[ \F g(\theta,\xi)=\sum_{n\in\Z^\kappa}g(n,\xi)e^{-2\pi i n\cdot \theta} \]
\[ \tilde{\F}f(n,\xi)=\int_{\R^d}f(n,x)e^{-2\pi i x\cdot\xi}dx  . \]
If we treat the operator $\F$ as the partial Fourier transform with respect to the first coordinate and $\tilde{F}$ the corresponding transform for the second coordinate, then we can say $\F\circ\tilde{\F}=\tilde{\F}\circ \F$ (even though the operators on the left and right are not precisely the same).

\begin{lemma} \label{liftedconst=}Let $d,\kappa\ge 1$, and $p\in(1,2)$, $q=p'$. The optimal constant ${\bf{A}}(q,d,\kappa)$ in the inequality 
\begin{equation}
    \|\widehat{f}\|_q\le {\bf{A}}(q,d,\kappa) |E|^{1/p}, \label{discrete}
\end{equation} 
where $E\subset{\Z^\kappa\times\R^d}$ satisfies $|E|<\infty$ and $|f|\le 1_E$, satisfies 
\[ {\bf{A}}(q,d,\kappa)={\bf{B}}_{q,d}.\]
The optimal constant ${\bf{A'}}(q,d,\kappa)$ for the inequality
\[ \|\widehat{f}\|_{q,*}\le {\bf{A'}}(q,d,\kappa)\|f\|_p\]
for $\Z^\kappa\times\R^d$ likewise satisfies ${\bf{A'}}(q,d,\kappa)={\bf{B}}_{q,d}$. 
\end{lemma}

\begin{proof}[Proof of Lemma \ref{liftedconst=}]
We analyze the mixed $L^p$ norms $L_n^pL_\xi^q(\Z_n^\kappa\times\R_\xi^d)$ and $L_\xi^qL_n^p(\Z_n^\kappa\times \R_\xi^d)$, given respectively by 
\[ \|g\|_{L_n^pL_\xi^q}=\left(\sum_n\left(\int|g(n,\xi)|^qd\xi\right)^{p/q}\right)^{1/p}\quad\text{and}\quad
\|g\|_{L_\xi^qL_n^p}=\left(\int\left(\sum_n|g(n,\xi)|^p\right)^{q/p}d\xi\right)^{1/q}. \]
There are corresponding norms for $L_\theta^sL_x^t(\T_\theta^\kappa\times\R_x^d)$ and $L_x^tL_\theta^s(\T_\theta^\kappa\times\R_x^d)$. Since $q\ge p$, we have by Minkowski's integral inequality that
\[ \|g\|_{L_\theta^qL_x^p(\T^\kappa\times\R^d)}\le \|g\|_{L_x^pL_\theta^q(\T^\kappa\times\R^d)}. \]
If $\f{F}$ denotes the Fourier transform from $\Z^\kappa$ to $\T^\kappa$ defined by 
\[ \f{F}g(\theta)=\sum_ng(n)e^{-2\pi in\cdot\theta},\]
then the optimal constant in the corresponding Hausdorff-Young inequality for $p\in(1,2)$ is 1. Thus if $|f|\le 1_E$ for $E\subset{\Z^\kappa}$ and $|E|<\infty$, we have
\begin{equation}
    \|\f{F}f\|_q\le \|f\|_p\le |E|^{1/p} \label{zhy}.
\end{equation}
This means that for $g\in L^q_\xi L^p_n(\Z_n^\kappa\times\R^d_\xi)$,
\begin{align*}
    \|\F g\|_{L_\xi^qL_\theta^q}&= \left(\int\int|\F g(\theta,\xi)|^qd\theta d\xi\right) ^{1/q} \\
    &\le \left(\int\left(\sum_n|g(n,\xi)|^p\right)^{q/p} d\xi\right) ^{1/q} ,
\end{align*}
so $\mc{F}$ is a contraction from $ L^q_\xi L^p_n(\R^d_\xi\times\Z_n^\kappa)$ to $L_\xi^qL_\theta^q(\R_\xi^d\times\T_{\theta}^\kappa)$. 

Let $|f|\le 1_E\in L^p(\Z^\kappa\times\R^d)$. For $n\in\Z^\kappa$, define the subset $E_n\subset\R^d$ and the function $f_n:\R^d\to\C$ by 
\begin{align} 
E_n=\{x\in\R^d:(n,x)\in E\}\label{defEn}\\
f_n(x)=f(n,x),\label{deffn}
\end{align} 
noting that $f_n\in L^p(\R^d)$. Since $|f_n|\le 1_{E_n}$, 
\begin{align} 
    \|\tilde{\F}f\|_{L^p_nL_\xi^q}&= \left(\sum_n\left(\int |\tilde{\F}f(n,\xi)|^qd\xi\right)^{p/q}\right)^{1/p}\nonumber \\ 
    & = \left(\sum_n\left(\int \left|\int f_n(x)e^{-2\pi i x\cdot\xi} dx\right|^qd\xi\right)^{p/q}\right)^{1/p}\nonumber \\
    &\le \left(\sum_n{\bf{B}}_{q,d}^p|E_n|\right)^{1/p}={\bf{B}}_{q,d}|E|^{1/p}.\label{eq:sliceineq}
\end{align}

Combining the above inequalities yields 
\begin{equation}\label{eq:3ineqs}
    \|\widehat{f}\|_{L^q(\Z^\kappa\times\R^d)}=\|\F\tilde{\F}f\|_{L_\xi^qL_\theta^q}\le \|\tilde{\F}f\|_{L^q_\xi L^p_n}\le \|\tilde{\F}f\|_{L_n^pL_\xi^q}\le {\bf{B}}_{q,d}|E|^{1/p}  ,
\end{equation} 
where we use (\ref{eq:sliceineq}) in the last inequality. Thus ${\bf{A}}(q,d,\kappa)\le {\bf{B}}_{q,d}$. Now let $|f|\le 1_E\in L^p(\R^d)$ be given. Define $E_0=\{0\}\times E$ and $f_0:\Z^\kappa\times\R^d \to \C$ by $f_0(n,x)=0$ for $n\not=0$ and $f_0(0,x)=f(x)$. Let $\tilde{\f{F}}$ denote the Fourier transform on $\R^d$ defined by $\tilde{\f{F}}g(\xi)=\int g(x)e^{-2\pi i x\cdot\xi} dx$. Then 
\begin{align*} 
\|\tilde{\f{F}}f\|_{L^q(\R^d)}&=\|\F{f_0}(0,\cdot)\|_{L^q_\xi} \\
&=\| \F{\tilde{f}}\|_{L^q_\xi L_n^q} \\
&\le {\bf{A}}(q,d,\kappa) |E_0|^{1/p}={\bf{A}}(q,d,\kappa)|E|^{1/p}.
\end{align*} 
This yields the reverse inequality ${\bf{A}}(q,d,\kappa)\ge {\bf{B}}_{q,d}$.

Now consider ${\bf{A}}'(q,d,\kappa)$. Let $f\in L^p(\Z^\kappa\times\R^d)$ and let $E\subset\Z^k\times\R^d$ be a Lebesgue measurable set satisfying $|E|\in\R^+$. Writing $E_\theta=\{\xi:(\theta,\xi)\in E\}$, 

\begin{align*}
\int_E|\widehat{f}|&= \int_{\T^\kappa}\int_{\R^d}|\mc{F}\tilde{\mc{F}}f(\theta,\xi)|1_E(\theta,\xi)d\xi d\theta \\
    &\le   \int_{\T^\kappa}\|\tilde{\mc{F}}\mc{F}f(\theta,\cdot)\|_{L^{q,*}_\xi}|E_{\theta}|^{1/p} d\theta \\
    &\le   \int_{\T^\kappa}{\bf{B}}_{q,d}\|{\mc{F}}f(\theta,\cdot)\|_{L^p_x}|E_{\theta}|^{1/p} d\theta \\
    &\le{\bf{B}}_{q,d} \left(   \int_{\T^\kappa}\|{\mc{F}}f(\theta,\cdot)\|^q_{L^p_x}d\theta \right)^{1/q}  \left(\int_{\T^\kappa}|E_{\theta}|^{p/p} d\theta\right)^{1/p}\\
    &\le{\bf{B}}_{q,d} \|{\mc{F}}f\|_{L^p_xL^q_\theta}|E|^{1/p} \\
    &\le{\bf{B}}_{q,d} \|f\|_{L^p}|E|^{1/p},
\end{align*}
so ${\bf{A}}'(q,d,\kappa)\le {\bf{B}}_{q,d}$. For the reverse inequality, let $f\in L^p(\R^d)$ and let $E\subset\R^d$ be a Lebesgue measurable set with $|E|\in\R^+$. Let $f_0$ and $E_0$ be defined as above. Then 
\begin{align*} 
\int_{E}|\tilde{\f{F}}f|&=\sum_n\int|\F{f_0}(n,\xi)|1_{E_0}(n,\xi)d\xi \\
&\le {\bf{A}}'(q,d,\kappa)|E_0|^{1/p}\|f_0\|_{L^p(\Z^\kappa\times\R^d)} \\
&= {\bf{A}}'(q,d,\kappa) |E|^{1/p}\|f\|_{L^p(\R^d)}. 
\end{align*} 
\end{proof}

In the remainder of the subsection, we prove the following two propositions concerning the structure of near-extremizers of the sharp Hausdorff-Young inequality on $\Z^\kappa\times\R^d$. 

\begin{proposition}\label{near-ext=slice} Let $d,\kappa\ge 1$ and $q\in(2,\infty)$, $p=q'$. Let $\delta>0$ be small. Let $0\not=f\in L^{p}(\Z^\kappa\times\R^d)$, $|f|\le 1_E$ where $E\subset\Z^\kappa\times\R^d$ is Lebesgue measurable and $|E|<\infty$. If $\|\widehat{f}\|_{q}\ge (1-\delta){\bf{B}}_{q,d} |E|^{1/p}$, then there exists $m\in\Z^\kappa$ such that 
\begin{equation}\label{eq:slice}
 |E_m|\ge (1-o_\delta(1))|E|
\end{equation}
where $E_m$ is defined in (\ref{defEn}).
\end{proposition}

The analogous proof of Proposition \ref{near-ext=slice} for the dual inequality fails to go through, which is the reason our results are partial. This leads to the following question, which is left open. Our final precompactness result is conditional on a positive answer to this question. 
\begin{question} \label{dual near-ext=slice} Let $d,\kappa\ge 1$ and $q\in(2,\infty)$, $p=q'$. Let $\delta>0$ be small. Let $0\not=f\in L^{p}(\Z^\kappa\times\R^d)$. If $\|\widehat{f}\|_{q,*}\ge (1-\delta){\bf{B}}_{q,d} \|f\|_p$, then must there exist $m\in\Z^\kappa$ such that 
\begin{equation}\label{eq:dualslice}
\|f_m\|_{L^p(\R^d)}\ge (1-o_\delta(1))\|f\|_{L^p(\Z^\kappa\times\R^d)},
\end{equation}
where $f_m$ is defined in (\ref{deffn})?
\end{question}

In the analysis of ${\bf{A}}(q,d,\kappa)$ from Lemma \ref{liftedconst=}, we proved a string of inequalities in (\ref{eq:3ineqs}). Combining these inequalities with the assumption that $(f,E)$ are $\delta$-near extremizing yields the following lemma, which requires no further proof. 

\begin{lemma}\label{mixedineqs} Let $d,\kappa\ge 1$ and $q\in(2,\infty)$. Set $p=q'$. Let $\delta>0$, let $E\subset{\Z^\kappa\times\R^d}$ be a Lebesgue measurable set with $|E|\in\R^+$, and let $f$ be a measurable function satisfying $|f|\le1_E$. If $\|\widehat{f}\|_{q}\ge(1-\delta){\bf{B}}_{q,d}|E|^{1/p}$, then all of the following hold: 
\begin{align}
    \label{ineq1} \|\F\tilde{\F}f\|_{L^q_\xi L^q_\theta}&\ge (1-\delta)\|\tilde{\F}f\|_{L^p_n L^q_\xi}\\
    \label{ineq2} \|\tilde{\F}f\|_{L^q_\xi L^p_n}&\ge(1-\delta)\|\tilde{\F}f\|_{L_n^p L_\xi^q} \\
    \label{ineq3} \|\tilde{\F}f\|_{L_n^p L_\xi^q}&\ge (1-\delta){\bf{B}}_{q,d}|E|^{1/p}
\end{align}
\end{lemma}

 The inequalities listed in Lemma \ref{mixedineqs} will be used to establish the following weak result, which is a preliminary for showing that any near extremizer of the lifted problem is mostly supported on one slice of the $\Z^\kappa$ variable. 

\begin{lemma}\label{weakerslice} Let $E\subset\Z^\kappa\times\R^d$ and $|f|\le 1_E$ satisfy $\|\widehat{f}\|_q\ge(1-\delta){\bf{B}}_{q,d}|E|^{1/p}$. There exists a disjointly supported decomposition 
\[ \tilde{\F}f(n,\xi)=g(n,\xi)+h(n,\xi) \]
where 
\[ \|h\|_{L^q_\xi L_n^p}\le o_\delta(1)|E|^{1/p}\]
and for each $\xi\in\R^d$ there exists $n(\xi)\in\Z^\kappa$ such that
\[ g(n,\xi)=0\quad \text{for all }n\not=n(\xi).\]
\end{lemma}
\begin{proof}[Proof of Lemma \ref{weakerslice}] This is completely analogous to the proof of Lemma 10.14 in \cite{c2}. 

Let $\eta=\delta^{1/2}$. Since $|f|\le 1_E$, for each $n\in\Z^\kappa$ the function $\tilde{\F}f(n,\xi)$ is a continuous function of $\xi$. Thus $\p_\xi(n):=\tilde{\F}f(n,\xi)$ is well-defined for every $\xi\in\R^d$. Define 
\[ \mc{G}=\{\xi\in\R^d:\p_\xi\not=0,\quad \|\widehat{\p_\xi}\|_{ L^q_\theta}\ge (1-\eta)\|\p_\xi\|_{ L^p_n}\}.\]
Here, $\widehat{\cdot}$ denotes the Fourier transform for $\Z^\kappa$. Then 
\begin{align*}
    \|\F\tilde{\F}f\|^q_{L^q_\theta L^q_\xi}&= \int_{\R^d\setminus\mc{G}}\|\widehat{\p_{\xi}}\|^q_{L^q_\theta}d\xi+\int_{\mc{G}}\|\widehat{\p_{\xi}}\|^q_{L^q_\theta}d\xi \\
    &\le (1-\eta)^q\int_{\R^d\setminus\mc{G}}\|\p_{\xi}\|^q_{L^p_n}d\xi+\int_{\mc{G}}\|\p_{\xi}\|^q_{L^p_n}d\xi \\
    &\le \int_{\R^d}\|\tilde{\F}f\|^q_{L^p_n}d\xi-c\eta\int_{\R^d\setminus\mc{G}}\|\tilde{\F}f\|^q_{L^p_n}d\xi.
\end{align*}
Combining this with (\ref{ineq1}), we get
\begin{align*}
    (1-\delta)^q\|\tilde{\F}f\|_{L^p_n L^q_\xi}^q&\le\|\F\tilde{\F}f\|_{L^q_\xi L^q_\theta}^q \\
    &\le \int_{\R^d}\|\tilde{\F}f\|^q_{L^p_n}d\xi-c\eta\int_{\R^d\setminus\mc{G}}\|\tilde{\F}f\|^q_{L^p_n}d\xi. 
\end{align*}
Rearranging the above inequality, obtain 
\begin{equation}\label{offGbound}
\int_{\R^d\setminus\mc{G}}\|\tilde{\F}f\|^q_{L^p_n}d\xi \le c'\delta^{1/2} \|\tilde{\F}f\|_{L^p_n L^q_\xi}^q.
\end{equation} 
For each $\xi\in\mc{G}$, $\|\widehat{\p_\xi}\|_{L^q_\theta}\ge(1-\eta)\|\p_\xi\|_{L^p_n}$, so we can invoke the argument beginning in line (7) of \cite{fournier} or Theorem 1.3 from \cite{cc} to get $n=n(\xi)\in\Z^\kappa$ such that 
\[ \|\p_\xi\|_{L^p(\Z^k\setminus\{n(\xi)\})}\le o_{\eta}(1)\|\p_\xi\|_{L^p(\Z^\kappa)}. \]
Define 
\[ g(n,\xi)=\begin{cases} \quad \p_\xi(n)\quad&\text{if}\quad n=n(\xi),\,\xi\in\mc{G}\\ \quad 0\quad&\text{else.}\quad 
\end{cases}\]
Let $h(n,\xi):=\tilde{\F}f(n,\xi)-g(n,\xi)$. Note that $g$ satisfies the conclusions of the lemma by its definition. To bound $\|h\|_{L^q_\xi L^p_n}$, we use the definition of $g$ as well as (\ref{offGbound}) to get 
\begin{align*}
    \|h\|_{L^q_\xi L^p_n}^q&\le \int_{\mc{G}}\|\tilde{\mc{\F}}f-g\|_{L^p_n}^qd\xi+ \int_{\R^d\setminus\mc{G}}\|\tilde{\mc{\F}}f\|_{L^p_n}^qd\xi\\
    &= \int_{\mc{G}} \left(\|\tilde{\F}f\|_{L^p(\Z^\kappa\setminus{n(\xi)})}^p+|\tilde{\F}f(n(\xi),\xi)-g(n(\xi),\xi)|^p\right)^{q/p}d\xi +\int_{\R^d\setminus\mc{G}}\|\tilde{\F}f\|^q_{L^p_n}d\xi  \\
    &= \int_{\mc{G}} \left(\|\tilde{\F}f\|_{L^p(\Z^\kappa\setminus{n(\xi)})}^p+0\right)^{q/p}d\xi +\int_{\R^d\setminus\mc{G}}\|\tilde{\F}f\|^q_{L^p_n}d\xi  \\
    &\le \int_{\mc{G}} \left(o_\eta(1)\|\tilde{\F}f\|_{L^p(\Z^\kappa)}\right)^qd\xi +c'\delta^{1/2}\|\tilde{\F}f\|^q_{L^q_\xi L^p_n} = o_\delta(1)\|\tilde{\F}f\|^q_{L^q_\xi L^p_n}.
\end{align*}
\end{proof}

\begin{proof}[Proof of Proposition \ref{near-ext=slice}] 
Let $\tilde{\F}f=g+h$ as in Lemma \ref{weakerslice}. Combining $\|h\|_{L^q_\xi L^p_n}\le o_\delta(1)(1)|E|^{1/p}$ with (\ref{ineq3}) implies $\|h\|_{L^q_\xi L^p_n}\le o_\delta(1)\|\tilde{\F}f\|_{L^p_nL^q_\xi}$. Using this with (\ref{ineq1}) gives 
\begin{align}
    \|g\|_{L^q_\xi L^p_n}+\|h\|_{L^q_\xi L^p_n}\ge \|\F\tilde{\F}f\|_{L^q_\xi L^q_\theta}&\ge (1-\delta)\|\tilde{\F}f\|_{L^p_n L^q_\xi}\nonumber,
\end{align}
from which we conclude 
\begin{align}
    \|g\|_{L^q_\xi L^p_n}&\ge (1-o_\delta(1))\|\tilde{\F}f\|_{L^p_n L^q_\xi} \nonumber.
\end{align}
Noting that $\|g\|^q_{L^q_\xi L^p_n}=\|g\|^q_{L^q_\xi L^q_n}$, we further have 
\begin{align}  \|g\|_{L^q_\xi L^q_n}=\|g\|_{L^q_\xi L^p_n}\ge (1-o_\delta(1))\|\tilde{\F}f\|_{L^p_n L^q_\xi}\ge (1-o_\delta(1))\|g\|_{L^p_n L^q_\xi} .\label{qpgepq} \end{align}
Let $M=\sup_{n}\|g(n,\cdot)\|^q_{L^q_\xi}$ (which is finite by (\ref{qpgepq})) and calculate using (\ref{qpgepq})
\begin{align*}
    M^{\frac{q-p}{pq}}\left(\int|g(n(\xi),\xi)|^{q}d\xi\right)^{1/q}&\ge (1-o_\delta(1))M^{\frac{q-p}{pq}}\left(\sum_n\left(\int|g(n,\xi)|^qd\xi\right)^{p/q}\right)^{1/p}\\
    &\ge (1-o_\delta(1))\left(\sum_n\int|g(n,\xi)|^qd\xi\right)^{1/p} \\
    &= (1-o_\delta(1))\left(\int|g(n(\xi),\xi)|^qd\xi\right)^{1/p}
\end{align*}
and therefore 
\[ M\ge (1-o_\delta(1))^{\frac{pq}{q-p}}\left(\int|g(n(\xi),\xi)|^qd\xi\right)^{\left(\frac{1}{p}-\frac{1}{q}\right)\left(\frac{pq}{q-p}\right)} =(1-o_\delta(1))\int|g(n(\xi),\xi)|^qd\xi .\]
Thus there exists $n\in\Z^\kappa$ such that 
\[\int |g(n,\xi)|^qd\xi\ge (1-o_\delta(1))(\|\tilde{\F}f\|_{L^q_\xi L^p_n}-\|h\|_{L^q_\xi L^p_n})^{q}\ge (1-o_\delta(1)){\bf{B}}_{q,d}^q|E|^{q/p}.\]
Then 
\[{\bf{B}}_{q,d}^q |E_n|^{q/p}\ge \int |g(n,\xi)|^qd\xi\ge  (1-o_\delta(1)){\bf{B}}_{q,d}^q|E|^{q/p},\]
so $|E_n|^{1/p}\ge(1-o_\delta(1))|E|^{1/p}$.
\end{proof}

\end{subsection}

\begin{subsection}{Lifting to $\Z^\kappa\times\R^d$}

\begin{definition} Let $\mc{Q}_d=[-\frac{1}{2},\frac{1}{2}]^d$. To any function $f:\R^d\to \C$, associate the function $f^\dagger:\Z^d\times\R^d\to\C$ defined by 
\[ f^\dagger(n,x)=\begin{cases}
f(n+x)\,\,&\text{if }x\in\mc{Q}_d\\
0 &\text{if }x\not\in\mc{Q}_d. \end{cases} \]
For a measurable set $E\subset\R^d$, let $E^\dagger$ be the set in $\Z^d\times\R^d$ defined by 
\[   E^\dagger=\{(n,x):n+x\in E\}  . \]
\end{definition}

We abuse the notation of $\widehat{\cdot}$ in the following lemmas: if $g:\R^d\to\C$, then $\widehat{g}(\xi)=\int_{\R^d}e^{-2\pi ix\cdot\xi}g(x)dx$ and if $g:\Z^d\times\R^d\to\C$, then $\widehat{g}(\theta,\xi)=\sum\limits_{n\in\Z^d}\int_{\R^d}e^{-2\pi in\cdot\theta}e^{-2\pi ix\cdot\xi}g(n,x)dx$. 
\begin{lemma}\label{lifttonearexts} Let $d\ge 1$ and $q\in(2,\infty)$, $p=q'$. Let $\delta,\eta>0$ be small. Let $E\subset\R^d$ be a Lebesgue measurable set with $|E|\in\R^+$. Suppose that
\[ \text{distance}(x,\Z^d)\le \eta\quad\text{for all }x\in E\]
and that for $|f|\le 1_E$, 
\[ \|\widehat{f}\|_{L^q(\R^d)}\ge (1-\delta){\bf{B}}_{q,d}|E|^{1/p}. \]
Then 
\[ \|\widehat{f^\dagger}\|_{L^q(\mb{T}^d\times \R^d)}\ge (1-\delta-o_\eta(1)){\bf{B}}_{q,d}|E^\dagger|^{1/p} .\]
\end{lemma} 
\begin{proof} The conclusion of Lemma 9.1 of \cite{c1} is that for some $C,\g\in\R^+$, we have 
\[  \left|\|\widehat{f^\dagger}\|_{L^q(\mb{T}^d\times\R^d)}-\|\widehat{f}\|_{L^q(\R^d)}\right|\le C\eta^\g\|f\|_{L^p(\R^d)} . \]
It follows that 
\begin{align*}  
\|\widehat{f}\|_{L^q(\mb{T}^d\times\R^d)}&\ge (1-\delta){\bf{B}}_{q,d}|E|^{1/p}-C\eta^\g\|f\|_p\ge(1-\delta){\bf{B}}_{q,d}|E|^{1/p}-C\eta^\g|E|^{1/p}\\
&= (1-\delta-o_\eta(1)){\bf{B}}_{q,d}|E^\dagger|^{1/p},
\end{align*}
where we used that $|E|=|E^\dagger|$.  
\end{proof}

The following lemma is analogous to the previous lemma. Ultimately, it is necessary to establish analogous results for the norm $\|\cdot\|_{q,*}$ because we will use it to translate localization properties of near-extremizers to the Fourier transforms of near-extremizers.

\begin{lemma} \label{lifttonearextsdual}Let $d\ge 1$ and $q\in(2,\infty)$, $p=q'$. Let $\delta,\eta>0$ be small. Let $0\not=f\in L^p(\R^d)$. Suppose that
\[ f\not=0\implies \text{distance }(x,\Z^d)\le \eta \]
and that 
\[ \|\widehat{f}\|_{L^{q,*}(\R^d)}\ge (1-\delta){\bf{B}}_{q,d}\|f\|_p. \]
Then 
\[ \|\widehat{f^\dagger}\|_{L^{q,*}(\mb{T}^d\times \R^d)}\ge (1-2\delta-o_\eta(1)){\bf{B}}_{q,d}\|f^\dagger\|_{L^p(\Z^d\times\R^d)} .\] 
\end{lemma}
\begin{proof} Let $\xi=n(\xi)+\a(\xi)$ where $n(\xi)\in \Z^d$ and $\a(\xi)\in[-\frac{1}{2},\frac{1}{2}]^d=\mc{Q}_d$. The proof of Lemma 9.1 in \cite{c1} demonstrates that 
\begin{equation}
    \label{eqn:lemma9.1} \|\widehat{f^\dagger}(\theta,n(\xi)+\a(\xi))-\widehat{f}(n(\xi)+\theta)\|_{L^q_{\theta,\xi}}\le o_\eta(1)\|f\|_p. 
\end{equation}

Let $E\subset{\R^d}$ be such that $ |E|^{-1/p}\int_E|\widehat{f}(\xi)|d\xi\ge (1-2\delta){\bf{B}}_{q,d}\|f\|_{L^p_x}$. 
Define the lifted set $\tilde{E}=\{(\theta,\xi)\in\T^d\times\R^d:\theta+n(\xi)\in E\}$. Using (\ref{eqn:lemma9.1}), we calculate
\begin{align*}
    \int_{\tilde{E}}|\widehat{f^\dagger}(\theta,\xi)|d\theta d\xi&\ge \int_{\tilde{E}}|\widehat{f}(n(\xi)+\theta)|d\theta d\xi -\int_{\tilde{E}}|\widehat{f^\dagger}(\theta,n(\xi)+\a(\xi))-\widehat{f}(n(\xi)+\theta)|d\theta d\xi\\
    &\ge \int_{E}|\widehat{f}(\xi)|d\xi-|\tilde{E}|^{1/p}\|\widehat{f^\dagger}(\theta,n(\xi)+\a(\xi))-\widehat{f}(n(\xi)+\theta)\|_{L^q_{\theta,\xi}} \\
    &\ge |\tilde{E}|^{1/p}(1-2\delta){\bf{B}}_{q,d}\|f\|_{L^p_x}-|\tilde{E}|^{1/p}o_\eta(1)\|f\|_{L^p_x}\\
    &= |\tilde{E}|^{1/p}(1-2\delta-o_\eta(1)){\bf{B}}_{q,d}\|f^\dagger\|_{L^p_{n,x}}.
\end{align*}
\end{proof}

Translating general near-extremizers of (\ref{eqn:main}) and (\ref{eq:dual}) to near-extremizers satisfying the hypotheses of the previous two lemmas respectively will be much easier with the following Proposition 5.2 from \cite{c1}, stated here for the reader's convenience.  

\begin{proposition}(Approximation by $\Z^d$). \label{approxbyZ^d}For each $d\ge 1$ and ${\bf{r}}\ge 0$ there exists $c>0$ with the following property. Let $P$ be a continuum multiprogression in $\R^d$ of rank ${\bf{r}}$, whose Lebesgue measure satisfies $|P|=1$. Let $\delta\in(0,\frac{1}{2}]$. There exists $\mc{T}\in{\text{Aff}(d)}$ whose Jacobian determinant satisfies 
\[ |\det J(\mc{T})|\ge c\delta^{d{\bf{r}}+d^2}\]
such that 
\[ \|\mc{T}(x)\|_{\R^d/\Z^d}<\delta\quad \text{for all }x\in P. \]

\end{proposition}

\end{subsection}

\begin{subsection}{Spatial localization}

\begin{proposition} \label{basicallyanellipse}Let $d\ge 1$ and $q\in(2,\infty)$, $p=q'$. For every $\epsilon>0$ there exists $\delta>0$ with the following property. Let $E$ be a measurable set with $|E|\in\R^+$ and $|f|\le 1_E$. If $\|\widehat{f}\|_{q}\ge (1-\delta){\bf{B}}_{q,d}|E|^{1/p}$, then there exists an ellipsoid $\mc{E}\subset{\R^d}$ satisfying
\begin{align}
\label{mostlyonmcE}    |E\setminus\mc{E}|&\le\epsilon|E| \\
\label{mcEiscontrolled}    |\mc{E}|&\le C_\epsilon|E|.
\end{align}
\end{proposition}
\begin{proof} Assume that $|E|^{1/p}{\bf{B}}_{q,d}(1-\delta)\le \|\widehat{f}\|_q$, where $\delta$ is to be chosen below. 

\begin{enumerate}
    \item Using the structural lemma for near extremizers of (\ref{eqn:main}), Lemma \ref{nearextstructure} with $\epsilon_0>0$ to be chosen later, we obtain a decomposition $E=A\cup B$ and a multiprogression $P$ satisfying 
    \[ E=A\cup B, \quad A\cap B=\emptyset,\]
    \[ |B|\le \epsilon_0|E|,\]
    \[ |P|\le C_{\epsilon_0}|E|,\]
    \[ A\prec P,\]
    \[ \text{rank }P\le C_{\epsilon_0}.\]
    
    \item By precomposing $f$ with an affine transformation, assume without loss of generality that $|P|=1$. Then for a fixed $\delta_0\in(0,\frac{1}{2}]$ to be chosen below, Proposition 5.2 in \cite{c1}, otherwise known as Proposition \ref{approxbyZ^d} in this paper, allows us to find a $c=c(d,p)$ as well as $\mc{T}\in{\text{Aff}(d)}$ such that 
    \[ |\det J(\mc{T})|\ge c\delta_0^{d C_{\epsilon_0}+d^2}\quad\text{ and}\]
    \[ \|\mc{T}(A)\|_{\R^d/\Z^d}<\delta_0\]
    where $J(\mc{T})$ is the Jacobian matrix of $\mc{T}$.
    \item Now taking $\eta_0=\delta_0$ in the hypothesis of Lemma \ref{lifttonearexts}, we are guaranteed that since 
    \[ \|\widehat{1_Af\circ\mc{T}^{-1}}\|_q\ge (1-\delta){\bf{B}}_{q,d}|\mc{T}^{-1}(E)|^{1/p}-\|\widehat{1_Bf}\|_q\ge (1-\delta-o_{\epsilon_0}(1)){\bf{B}}_{q,d}|A|^{1/p}\]
    and $\|\mc{T}(A)\|_{\R^d/\Z^d}<\delta_0$, 
    we have
    \[ \|\widehat{(1_Af\circ\mc{T}^{-1})^{\dagger}}\|_{L^q(\T^d\times\R^d)}\ge (1-\delta-o_{\epsilon_0}(1)-o_{\delta_0}(1)){\bf{B}}_{q,d}|\mc{T}(A)^\dagger|^{1/p}, \]
    where $\widehat{\cdot}$ here denotes the Fourier transform on $\Z^d\times\R^d$.
    
    \item  Then Proposition \ref{near-ext=slice} gives the existence of $m\in\Z^d$ such that 
    \[ |\mc{T}(A)\cap(m+[1/2,1/2)^d)|\ge (1-o_\delta(1)-o_{\epsilon_0}(1)-o_{\delta_0}(1))|\mc{T}(A)|. \]
    
    \item Last, we note that the cube $Q:=m+[1/2,1/2)^d$ satisfies 
    \begin{align*}  
    |E\setminus\mc{T}^{-1}(Q)|&\le|A|+|B|-|A\cap\mc{T}^{-1}(Q)|\\
    &\le |A|+\epsilon_0|E|-(1-o_{\delta}(1)-o_{\epsilon_0}(1)-o_{\delta_0}(1))|A|\\
    &\le (\epsilon_0+o_{\delta}(1)+o_{\epsilon_0}(1)+o_{\delta_0}(1))|E|. 
    \end{align*} 
    Note that $\epsilon_0$ and $\delta_0$ may be chosen freely, and $\delta$ may be taken small enough after fixing an $\epsilon_0$ and $\delta_0$. Thus we may choose $\epsilon_0$, $\delta_0$, and then $\delta$ small enough so that $|E\setminus\mc{T}^{-1}(Q)|\le \epsilon|E|$. 
    We also note that 
    \begin{align*}
    |\mc{T}^{-1}(Q)|&=|\det J(\mc{T})|^{-1}|Q| \\
    &= |\det J(\mc{T})|^{-1}|P|\\
    &\le (c\delta_0^{dC_{\epsilon_0}+d^2})^{-1}C_{\epsilon_0}|E|\\
    &= \tilde{C}_\epsilon|E|. 
    \end{align*}
    Finally, since $Q$ is comparable in size (up to dimensional constants) to the smallest ball which contains it, we are done. 
\end{enumerate}

\end{proof} 

\begin{proposition} \label{ellipsoiddecompdual} Suppose that the claim in Question \ref{dual near-ext=slice} holds. Let $d\ge 1$ and $q\in(2,\infty)$, $p=q'$. For every $\epsilon>0$ there exists $\delta>0$ with the following property. Let $0\not=f\in L^{q'}(\R^d)$ satisfy $\|\widehat{f}\|_{q,*}\ge (1-\delta){\bf{B}}_{q,d}\|f\|_{p}$. There exists an ellipsoid $\mc{E}\subset{\R^d}$ and a decomposition $f=\phi+\psi$ such that
\[ \|\psi\|_{q'}<\epsilon \|f\|_{p} \]
\[ \phi\equiv 0\quad\text{on }\R^d\setminus\mc{E}  \]
\[ \|\phi\|_\infty|\mc{E}|^{1/p}\le C_\epsilon\|f\|_{p}.\]
\end{proposition}
\begin{proof} We follow an analogous argument as that in the proof of Proposition \ref{basicallyanellipse}, replacing the near extremizer structure Lemma \ref{nearextstructure} by the analogous structure theorem for the dual problem, Lemma \ref{nearextstructuredual}. For step (3) in the proof of Proposition \ref{basicallyanellipse}, we use Lemma \ref{lifttonearextsdual} in place of Lemma \ref{lifttonearexts}. For step (4), use an affirmative answer to Question \ref{dual near-ext=slice} instead of Proposition \ref{near-ext=slice}. The conclusion is that using analogous notation as in the proof of Proposition \ref{basicallyanellipse},

\[ \|1_Af\circ \mc{T}^{-1}1_Q\|_{L^p(\R^d)}\ge (1-o_\delta(1)-o_{\epsilon_0}(1)-o_{\delta_0}(1))\|1_Af\circ \mc{T}^{-1}\|_{L^p(\R^d)} \]
where $\epsilon_0$ and $\delta_0$ may be chosen freely, and $\delta$ may be taken small enough after fixing an $\epsilon_0$ and $\delta_0$. Let $\mc{E}$ be the smallest ellipsoid containing $\mc{T}^{-1}(Q)$ and define $\phi=1_{A\cap\mc{E}}f$, so $\psi=f-\p$. 
Then for small enough parameters $\delta_0,\epsilon_0$ and then $\delta$,
\[ \|\psi\|_p<\epsilon\|f\|_p \quad\text{and}\quad \p\prec 1_{\mc{E}}. \]
By the construction, we also have that 
\begin{align*}
    \|\p\|_{\infty}|\mc{E}|^{1/p}&\le c_d\|1_Af\|_{\infty}|\mc{T}^{-1}(Q)|\\
    &\le c_dC_{\epsilon_0}\|f\|_p |\det J(\mc{T}^{-1})| \\
    &\le c_dC_{\epsilon_0}\|f\|_p (c\delta_0^{dC_{\epsilon_0}+d^2})^{-1},
\end{align*}
so we are done. 
\end{proof}

\end{subsection}

\begin{subsection}{Frequency localization}

\begin{proposition} \label{basicallyanellipsedual} Suppose that the claim in Question \ref{dual near-ext=slice} holds. Let $d\ge 1$ and $q\in(2,\infty)$, $p=q'$. For every $\epsilon>0$ there exists $\delta>0$ with the following property. Let $E$ be a Lebesgue measurable set with $|E|\in\R^+$. Suppose that $|f|\le 1_E$ satisfies $\|\widehat{f}\|_{q}\ge (1-\delta){\bf{B}}_{q,d}|E|^{1/p}$. Then there exists an ellipsoid $\mc{E'}\subset{\R^d}$ and a decomposition $\widehat{f}=\Phi+\Psi$ such that
\[ \|\Psi\|_{q'}<\epsilon \|\widehat{f}\|_{p} \]
\[ \Phi\equiv 0\quad\text{on }\R^d\setminus\mc{E'}  \]
\[ \|\Phi\|_\infty|\mc{E'}|^{1/p}\le C_\epsilon\|f\|_{p}.\]
\end{proposition}

\begin{proof} In the proof of Proposition \ref{dualineq} we showed that if $(f,E)$ is a near-extremizing pair for (\ref{eqn:main}), then $\widehat{f}|\widehat{f}|^{q-2}$ is a near-extremizer for (\ref{eq:dualmain}). Thus we may apply Proposition \ref{ellipsoiddecompdual} to obtain a decomposition $\widehat{f}|\widehat{f}|^{q-2}=\p+\psi$ and take $\Phi=\p|\p|^{(2-q)/(q-1)}$ and $\Psi=\psi|\psi|^{(2-q)/(q-1)}$ for the desired decomposition. 
\end{proof}

\end{subsection}

\begin{subsection}{Compatibility of approximating ellipsoids}

We will show that $\mc{E}$ and $\mc{E}'$ are dual to each other, up to bounded factors and independent translations. For $s\in\R^+$ and $E\subset{\R^d}$, we consider the dilated set $sE=\{sy:y\in E\}$.

\begin{definition}\label{def:normalized}
The polar set $\mc{E}^*$ of a balanced, bounded, convex set with nonempty interior $\mc{E}\subset\R^d$ is 
\[ \mc{E}^*=\{y:\,|\langle x,y\rangle|\le 1\,\,\text{for every }x\in\mc{E}\} \]
where $\langle\cdot,\cdot\rangle$ denotes the Euclidean inner product. 
\end{definition}

\begin{lemma}\label{Domcompat} Suppose that the claim in Question \ref{dual near-ext=slice} holds. Let $d\ge1$ and let $\Lambda\subset{(1,2)}$ be a compact set. There exists $\eta_0>0$ such that the following property holds for $0<\eta<\eta_0$. Let $\eta>0$. Let $p\in\Lambda$ and let $q=p'$. Suppose $\|\widehat{f}\|_q\ge (1-\rho(\eta)){\bf{B}}_{q,d}|E|^{1/p}$ for a function $\rho:[0,1]\to \R^+$ where $\rho(\eta)\to 0$ as $\eta\to 0$ sufficiently fast so that there exists an ellipsoid $\mc{E}+u$ satisfying the conclusions of Proposition \ref{basicallyanellipse} with $\epsilon=\eta$ and an ellipsoid $\tilde{\mc{E}}+v$ and disjoint decomposition $\widehat{f}=\Phi+\Psi$ satisfying the conclusions of Proposition \ref{ellipsoiddecompdual} with $\epsilon=\eta$, where $\mc{E}$ and $\tilde{\mc{E}}$ are ellipsoids centered at the origin and $u,v\in\R^d$. Then there exists a constant $C=C(d,\Lambda,\eta)$ such that 
\[ \mc{E}\subset{C\tilde{\mc{E}}}^* \quad\text{and}\quad \tilde{\mc{E}}\subset{C\mc{E}^*} .\]

\begin{proof} By constants, we mean quantities which are permitted to depend on $d,\Lambda,\eta$. By replacing $f$ and $1_E$ with $e^{2\pi i x\cdot v}f(x+u)$ and $1_E(x+u)$ respectively, we may assume without loss of generality that $u,v=0$. By dilating $f$ and $E$ by $|\mc{E}|^{1/d}$, we may further assume that $|\mc{E}|=1$.

First, we will prove that $|\tilde{\mc{E}}|=|\mc{E}||\tilde{\mc{E}}|\le C$. We have assumed that 
\begin{equation} \label{assumptioneta}
    (1-\rho(\eta)){\bf{B}}_{q,d}|E|^{1/p}\le \|\widehat{f}\|_q,
\end{equation}  
and hence by Proposition \ref{basicallyanellipse} we know that
\[ |E\setminus\mc{E}|\le \eta|E|,\quad|\mc{E}|\le C_\eta|E| \]
and by Proposition \ref{ellipsoiddecompdual} that
\[ \|\Psi\|_q\le\eta|E|^{1/p},\quad \Phi\prec \tilde{\mc{E}},\quad \|\Phi\|_\infty|\tilde{\mc{E}}|^{1/q}\le C_\eta|E|^{1/p}. \]

Let $S_\a=\{\xi:|\widehat{f}(\xi)|\ge \a|\tilde{\mc{E}}|^{-1/q}\}$ and $\lambda_\eta=\{\xi:|\widehat{f}(\xi)|\le C_\eta |E|^{1/p}{|\tilde{\mc{E}}|}^{-1/q}\}$. We decompose the following integral as
\begin{align}
    \int_{\R^d}|\widehat{f}|^q d\xi&=\int_{\tilde{\mc{E}}^c}|\widehat{f}|^q d\xi+\int_{{\tilde{\mc{E}}}\cap\lambda_\eta^c}|\widehat{f}|^qd\xi \label{l2caplp} +\int_{\tilde{\mc{E}}\cap (\lambda_\eta\cap S_\a)}|\widehat{f}|^q d\xi+\int_{\tilde{\mc{E}}\cap (\lambda_\eta\cap S_\a^c)}|\widehat{f}|^q d\xi \nonumber \\
    &:=A+B+C+D.\nonumber
\end{align}
We will bound each integral defined above. First, we use the properties of the decomposition $\widehat{f}=\Phi+\Psi$ to note that 
\[ A=\int_{\tilde{\mc{E}}^c}|\Psi|^qd\xi\le \|\Psi\|_q^q\le \eta^q|E|^{q/p} .\]
Next, to control $B$, we use the property that $\widehat{f}=\Phi+\Psi$ is  a disjointly supported decomposition and $\|\Phi\|_{\infty}\le C_\eta|E|^{1/p}|\tilde{\mc{E}}|^{-1/q}$, so $\Phi=0$ a.e. on $\lambda_\eta^c$. Namely,  
\[B=\int_{\tilde{\mc{E}}\cap\lambda_\eta^c }|\Psi|^q d\xi \le\eta^q|E|^{q/p}.  \]
For $C$, we use that $|\widehat{f}|\le C_\eta|E|^{1/p}|\tilde{\mc{E}}|^{-1/q}$ on $\lambda_\eta$ to get
\[ C\le C_\eta^q|E|^{q/p}|\tilde{\mc{E}}|^{-1}|\tilde{\mc{E}}\cap S_\a|.  \]
Finally, we have for $D$ that
\[D \le|\tilde{\mc{E}}\cap\lambda_\eta|\a^q|\tilde{\mc{E}}|^{-1}\le\a^q.\]

Combining the upper bounds for $A,B,C,D$ with (\ref{assumptioneta}), we have
\begin{align*}  
(1-\rho(\eta))^q{\bf{B}}_{q,d}^q|E|^{q/p}&\le \int_{\R^d}|f|^qd\xi\\
&= A+B+C+D\\
&\le\eta^q|E|^{q/p}+\eta^q|E|^{q/p}\\
&\quad\quad+ C_\eta^q|E|^{q/p}|\tilde{\mc{E}}|^{-1}|\tilde{\mc{E}}\cap S_\a|+ \a^q
\end{align*}
Rearranging, we get 
\[ C_{\eta}^{-q}[(1-\rho(\eta))^q{\bf{B}}_{q,d}^q-   2\eta^q- \a^q|E|^{-q/p}]|\tilde{\mc{E}}|\le |\tilde{\mc{E}}\cap S_\a| \]
Finally, since $|\mc{E}|=1$ and $|\mc{E}|\le C_\eta|E|$, we have
\[  C_{\eta}^{-q}[(1-o_\eta(1))^q{\bf{B}}_{q,d}^q-   2\eta^q- \a^qC_\eta^{q/p}]|\tilde{\mc{E}}|\le |\tilde{\mc{E}}\cap S_\a|.\]
Choose $\a$ small enough so that 
\[  \frac{1}{2}C_{\eta}^{-q}[(1-o_\eta(1))^q{\bf{B}}_{q,d}^q-   2\eta^q]\le 
C_{\eta}^{-q}[(1-o_\eta(1))^q{\bf{B}}_{q,d}^q-   2\eta^q- \a^qC_\eta^{q/p}],\]
so $\a$ only depends on $\eta$. 
Thus for $c'=c'(\eta)>0$ and $\a=\a(\eta)$, we can conclude that
\[ c'|\tilde{\mc{E}}|\le |\tilde{\mc{E}}\cap S_\a| . \]

Since $|f|\le 1_E$ and $|E|<\infty$, $f$ is in $L^2$. Since $|\mc{E}|=1$, note that $|E|=|E\cap\mc{E}|+|E\setminus\mc{E}|\le 1+\eta|E|$, so we can assume $|E|\le 2$. Using these two observations, we have
\begin{align*}
    2\ge |E|\ge \|f\|_2^2&=\|\widehat{f}\|_2^2\ge \int_{S_\a}|\widehat{f}(\xi)|^2d\xi\ge \a^2|\tilde{\mc{E}}|^{-2/q}|S_\a|\ge \a^2c'|\tilde{\mc{E}}|^{1-2/q}, 
\end{align*} 
so $|\mc{E}||\tilde{\mc{E}}|=|\tilde{\mc{E}}|\le C'$ for $C'=C'(\eta)$.

Now assume via composition with an affine transformation that $\tilde{\mc{E}}=\mb{B}$ and that $\mc{E}=\{x:\sum_{j=1}^ds_j^{-2}x_j^2\le 1\}$. We wish to show that $\mb{B}\subset C\mc{E}^*$ and that $\mc{E}\subset C\mb{B}$, where $C$ is permitted to depend on $\eta$. We know from the earlier discussion that $|\mc{E}||\mb{B}|\le C_{\eta}$. Since $|\mc{E}|=c_d\prod_{j=1}^ds_j$, it remains to show that the smallest $s_i$, say $s_1$, is bounded below. Using the same notation as earlier, we note that
\[ \|\partial_{\xi_1}\widehat{1_{\mc{E}}f}\|_{q}\le2\pi  {\bf{B}}_{q,d}\|x_11_{\mc{E}}f\|_p\le 2\pi {\bf{B}}_{q,d} s_1|E|^{1/p}\]
and that
\begin{align*}
    \|\widehat{1_{\mc{E}}f}\|_{L^q(\mb{B})}&\ge \|\widehat{f}\|_{L^q(\mb{B})}-\|\widehat{1_{E\setminus\mc{E}}f}\|_q \\
    &\ge \|\widehat{f}\|_{L^q(\R^d)}-\|\Psi\|_q-{\bf{B}}_{q,d}|E\setminus\mc{E}|^{1/p} \\
    &\ge (1-\rho(\eta)){\bf{B}}_{q,d}|E|^{1/p}-\eta|E|^{1/p}-{\bf{B}}_{q,d}\eta^{1/p}|E|^{1/p} \\
    &\ge (1-o_\eta(1)){\bf{B}}_{q,d}|E|^{1/p}. 
\end{align*}
We also have 
\[ \||E|^{-1/p}\widehat{1_{\mc{E}}f}\|_q\le {\bf{B}}_{q,d} .\]

Thus we are in the situation where there are functions $h$ satisfying $\|\partial_{\xi_1}h\|_{L^q(\R^d)}\le 2\pi{\bf{B}}_{q,d}s_1(h)$ for a positive quantity $s_1$ associated to each $h$ and $\|h\|_{L^q(\mb{B})}>\frac{1}{2}{\bf{B}}_{q,d}>0$. If there are functions $h=|E|^{-1/p}\widehat{1_{\mc{E}}f}$ fitting the above regime and for which $s_1(h)\to 0$, then $\|h\|_{L^q(\R^d)}\to\infty$. Since we have the uniform upper bound $\||E|^{-1/p}\widehat{1_{\mc{E}}f}\|_q\le {\bf{B}}_{q,d}$, there must be a positive lower bound depending on $\eta$ for the values of $s_1$, which completes the proof. 

\end{proof} 

\end{lemma}

\end{subsection}

\section{Precompactness\label{precompactnesssect}} 

We restate Theorem \ref{precompactness} for the reader's convenience. 

\noindent{{\bf{Theorem 1.1}} \emph{ Suppose that the claim in Question \ref{dual near-ext=slice} holds. Let $d\ge1$ and $q\in(2,\infty)$, $p=q'$. Let $(E_\nu)$ be a sequence of Lebesgue measurable subsets of $\R^d$ with $|E_\nu|\in\R^+$ and let $f_\nu$ be Lebesgue measurable functions on $\R^d$ satisfying $|f_\nu|\le 1_{E_\nu}$. Suppose that $\lim_{\nu\to\infty}|E_\nu|^{-1/p}\|\widehat{f_\nu}\|_{q}={\bf{B}}_{q,d}$. Then there exists a subsequence of indices $\nu_k$, a Lebesgue measurable set $E\subset\R^d$ with $0<|E|<\infty$, a Lebesgue measurable function $f$ on $\R^d$ satisfying $|f|\le 1_E$, a sequence  $(T_\nu)$ of affine automorphisms of $\R^d$, and a sequence of vectors $v_\nu\in\R^d$ such that  
\[ \lim_{k\to\infty}\|e^{-2\pi i v_{\nu_k}\cdot x}f_{\nu_k}\circ T_{\nu_k}^{-1}-f\|_p=0\quad\text{and}\quad \lim_{k\to\infty}|T_{\nu_k}(E_{\nu_k})\Delta E|=0. \]
}}

In order to prove Theorem \ref{precompactness}, we first prove the following lemma. 
\begin{lemma}\label{precompactnessdual} Suppose that the claim in Question \ref{dual near-ext=slice} holds. Let $d\ge1$ and $q\in(2,\infty)$, $p=q'$. Let $(E_\nu)$ be a sequence of Lebesgue measurable subsets of $\R^d$ with $|E_\nu|\in\R^+$. Let $f_\nu$ be Lebesgue measurable functions satisfying $|f_\nu|\le 1_{E_\nu}$. Suppose that $\lim_{\nu\to\infty}|E_\nu|^{-1/p}\|\widehat{f_\nu}\|_{q}={\bf{B}}_{q,d}$. Then there exists a sequence of elements $T_\nu\in\text{Aff}(d)$ and vectors $v_\nu\in\R^d$ such that $|T_\nu(E_\nu)|$ is uniformly bounded and the sequence of functions  $(\widehat{g_\nu})$ where $g_\nu=e^{-2\pi iv_\nu\cdot x}f_\nu\circ T_\nu^{-1}$ is precompact in $L^q(\R^d)$. 
\end{lemma}

\begin{proof}[Proof of Lemma \ref{precompactnessdual}] Let $f_\nu$ and $E_\nu$ satisfy the hypotheses. Let $\epsilon_0=\min(\frac{1}{4},\eta_0)$ where $\eta_0$ is the threshold from Lemma \ref{Domcompat}. For each sufficiently large $\nu$, (1) there exists an ellipsoid $\mc{E}_{\nu}$ satisfying the conclusions of Proposition \ref{basicallyanellipse} with $\epsilon=\epsilon_0$ and (2) there exists an ellipsoid $\mc{F}_{\nu}$ and disjointly supported decomposition $f_\nu=\Phi_{\nu}+\Psi_{\nu}$ satisfying the conclusions of Proposition \ref{basicallyanellipsedual} with $\epsilon=\epsilon_0$. 

Let $u_{\nu},v_\nu\in\R^d$ be the centers of the $\mc{E}_\nu$ and $\mc{F}_\nu$ respectively. By replacing $f_\nu$ by $e^{-2\pi iv_\nu\cdot x}f_\nu(x+u_\nu)$ and $1_{E_\nu}$ by $1_{E_\nu-u_n}$, we may reduce to the case $u_\nu=v_\nu=0$. By composing $f_n$ and $1_{E_\nu}$ with an element of the general linear group on $\R^d$, we may reduce to the case in which $\mc{E}_n$ is the unit ball $\B$ of $\R^d$. Continue to denote these modified functions by $f_\nu$ and $1_{E_\nu}$.

For each $\epsilon>0$, there exists $N<\infty$ such that for each $\nu\ge N$, Proposition \ref{basicallyanellipse} associates to $(f_\nu,E_\nu)$ an ellipsoid $\mc{E}_{\nu,\epsilon}$ 
and Proposition \ref{basicallyanellipsedual} associates to $(f_\nu,E_\nu)$ an ellipsoid $\mc{F}_{\nu,\epsilon}$ and a disjointly supported decomposition $\widehat{f_\nu}=\Phi_{\nu,\epsilon}+\Psi_{\nu,\epsilon}$.

Symmetries of the inequality have been exploited to normalize so that $\mc{E}_\nu=\B$, so by Lemma \ref{Domcompat}, $\mc{F}_n$ are balls centered at the origin with radii comparable to $1$. We claim that this ensures corresponding normalizations for $\mc{E}_{\nu,\epsilon},\mc{F}_{\nu,\epsilon}$; $\epsilon$--dependent symmetries
are not needed. 

According to Proposition \ref{basicallyanellipse}, 
\begin{align*}
|E_\nu\setminus\B|&\le \epsilon_0|E_\nu|\quad\text{and}\quad |\B|\le C_0|E_\nu|\\
|E_\nu\setminus\mc{E}_{\nu,\epsilon}|&\le \epsilon|E_\nu|\quad\text{and}\quad |\mc{E}_{\nu,\epsilon}|\le C_\epsilon|E_\nu|,     
\end{align*}
provided that $\nu$ is sufficiently large and $\epsilon$ is sufficiently small. This implies that
\begin{align*}
    \frac{3}{4}C_0^{-1}|\B|\le (1-\epsilon_0)|E_\nu|&\le  |\B\cap E_\nu|=|(\B\cap E_\nu)\setminus \mc{E}_{\nu,\epsilon}|+|\B\cap E_\nu\cap \mc{E}_{\nu,\epsilon}|\\
    &\le \epsilon |E_\nu|+|\B\cap \mc{E}_{\nu,\epsilon}|\le \epsilon \frac{4}{3}|\B|+|\B\cap \mc{E}_{\nu,\epsilon}| ,
\end{align*}
so there is a $c>0$ such that $|\B\cap \mc{E}_{\nu,\epsilon}|\ge c$ where $c$ is independent of $\epsilon$ and $\nu$. This lower bound combined with the upper bound $|\mc{E}_{\nu,\epsilon}|\le C_\epsilon |E_\nu|\le C_\epsilon\frac{4}{3}|\B|$ implies the $\mc{E}_{\nu,\epsilon}$ are contained in a ball centered at $0$ with radius depending only on $\epsilon$.

By Proposition \ref{basicallyanellipsedual}, for sufficiently large $\nu$ and sufficiently small $\epsilon$,
\begin{align*}
\|\Phi_\nu-\widehat{f_\nu}\|_q\le \epsilon_0\|\widehat{f_\nu}\|_q\quad\text{and}\quad \|\Phi_{\nu,\epsilon}-\widehat{f_\nu}\|_q\le\epsilon\|\widehat{f_\nu}\|_q,\\ 
\|\Phi_\nu\|_\infty\le C_0\|f_\nu\|_p \quad\text{and}\quad\|\Phi_{\nu,\epsilon}\|_\infty|\mc{E}_{\nu,\epsilon}|^{1/q}\le C_\epsilon\|f_\nu\|_p .
\end{align*}

For each $\xi\in\R^d$, each of $\Phi_\nu(\xi),\Phi_{\nu,\epsilon}(\xi)$ is equal either to $\widehat{f}_\nu(\xi)$, or to $0$. From these inequalities and this fact, along with the support relations $\Phi_\nu\prec\B$
and $\Phi_{\nu,\epsilon}\prec \mc{F}_{\nu,\epsilon}$, it follows that 
\begin{align*}
    \|\Phi_{\nu,\epsilon}\|_{q}&=\|1_{\mc{F}_{\nu,\epsilon}}\widehat{f_\nu}\|_q\ge (1-\epsilon)\|\widehat{f_\nu}\|_q \\
    \|1_{\mc{F}_{\nu,\epsilon}\setminus C_0\B}\widehat{f_\nu}\|_q&\le \|\Phi_{\nu}-\Phi_{\nu,\epsilon}\|_q\le (\epsilon_0+\epsilon)\|\widehat{f_\nu}\|_q 
\end{align*}
where $\mc{F}_{\nu}\subset C_0\B$. Thus $\|1_{\mc{F}_{\nu,\epsilon}\cap C_0\B}\widehat{f_\nu}\|_q\ge (1-\epsilon_0-2\epsilon)\|\widehat{f_\nu}\|_q$. Combined with the inequalities 
\begin{align*} 
\|\widehat{f_\nu}\|_q \ge \frac{1}{2}{\bf{B}}_{q,d}|E_\nu|^{1/p}&\ge \frac{1}{2}{\bf{B}}_{q,d}C_0^{-1/p}|\B|^{1/p} \quad\text{and}\\
 \|1_{\mc{F}_{\nu,\epsilon}\cap C_0\B}\widehat{f_\nu}\|_q\le  |\mc{F}_{\nu,\epsilon}\cap C_0\B|^{1/q}\|\widehat{f_\nu}\|_\infty&\le |\mc{F}_{\nu,\epsilon}\cap C_0\B|^{1/q}|E_\nu|\le  |\mc{F}_{\nu,\epsilon}\cap C_0\B|^{1/q}\frac{4}{3}|\B|, 
\end{align*}
we conclude $|\mc{F}_{\nu,\epsilon}\cap C_0\B|\ge c$ where $c>0$ is independent of $\nu,\epsilon$.

Another consequence of the inequalities from Proposition \ref{basicallyanellipsedual} is that $\|\Phi_{\nu,\epsilon}\|_\infty\ge c\|\Phi_{\nu}\|_\infty$, where $c>0$ is independent of $\nu,\epsilon$. Indeed, 
\begin{align*}
(1-\epsilon_0)\|\widehat{f_\nu}\|_q&\le \|\Phi_\nu\|_{q}=\|\Phi_\nu\|_{L^q(C_0\B)}\le \|\Phi_\nu-\Phi_{\nu,\epsilon}\|_q+\|\Phi_{\nu,\epsilon}\|_{L^q(C_0\B)}\\
    &\le (\epsilon_0+\epsilon)\|\widehat{f_\nu}\|_q+|\Phi_{\nu,\epsilon}\|_\infty|C_0\B|^{1/q}, 
\end{align*}
so for a constant $c>0$ independent of $\nu,\epsilon$, \[ \|\Phi_{\nu,\epsilon}\|_\infty\ge c\|\widehat{f_\nu}\|_q\ge c\frac{1}{2}{\bf{B}}_{q,d}|E_\nu|^{1/p}\ge c\frac{1}{2}{\bf{B}}_{q,d}\|f_\nu\|_p. \]
Then $|\mc{F}_{\nu,\epsilon}|^{1/q} \le C_\epsilon\|f_\nu\|_p \|\Phi_{\nu,\epsilon}\|_\infty^{-1}
\le C_\epsilon \frac{2}{c{\bf{B}}_{q,d}}$. The uniform lower bound on $|\mc{F}_{\nu,\epsilon}\cap C_0\B|$ and the $\epsilon$-dependent upper bound on $|\mc{F}_{\nu,\epsilon}|$ imply that $\mc{F}_{\nu,\epsilon}\subset C_\epsilon \B$ for sufficiently large $\nu$.

Now note the uniform bounds 
\[  \|\widehat{1_{C_0\B}f_\nu}\|_q\le {\bf{B}}_{q,d}|E_\nu|^{1/q}\le {\bf{B}}_{q,d}\frac{4}{3}|\B|^{1/p}   \] 
and, by the Hausdorff-Young inequality,
\[ \|\nabla\widehat{1_{C_0\B }f_\nu}\|_q\le \||x|1_{\mc{E}_{\nu,\epsilon}}\|_p\le C_\epsilon  \]
since $|x|1_{\mc{E}_{\nu,\epsilon}}$ is bounded by the diameter of $\mc{E}_{\nu,\epsilon}$ and the volumes $|\mc{E}_{\nu,\epsilon}|$ are bounded above uniformly in $\nu$. By Rellich's theorem, on any fixed bounded subset of $\R^d$, we can find an $L^q$ convergent subsequence of $(\widehat{1_{\mc{E}_{\nu,\epsilon}}f_\nu })$. Since this is true for each $\epsilon$, $\|\widehat{f_\nu}\|_q$ is bounded uniformly above, and $\|\vwidehat{1_{\R^d\setminus\mc{E}_{\nu,\epsilon}}f_\nu}\|_q\to 0$ as $\epsilon\to 0$, it follows that the sequence $(\widehat{f_\nu})$ is precompact in $L^q(\R^d)$ on any fixed bounded subset of $\R^d$. Since $\mc{F}_{\nu,\epsilon}$ is contained
in a ball independent of $\nu$ for each fixed $\epsilon$, and since 
$\int_{\xi\notin\mc{F}_{\nu,\epsilon}}|\widehat{f_\nu}(\xi)|^q\,d\xi\to 0$ as $\epsilon\to0$,
the sequence $(\widehat{f_\nu})$ is precompact in $L^q(\R^d)$.

\end{proof}

\begin{proof}[Proof of Theorem \ref{precompactness}] 

From Lemma \ref{precompactnessdual}, we can assume that the sequence $(\widehat{f_{\nu}})$ is convergent in $L^q(\R^d)$ and that the supports $E_\nu$ satisfy $|E_\nu|\le \frac{4}{3}|\B|$. By passing to a subsequence, we may also assume that $\lim_{\nu\to\infty}|E_{\nu}|=a$ where $0<a<\infty$ and by precomposing $f_\nu$ and $1_{E_\nu}$ by affine transformations, we may assume that $|E_\nu|=1$ for all $\nu$. Then
\[ \lim_{\nu\to\infty}\|\widehat{f_\nu}\|_q= {\bf{B}}_{q,d}. \]

Since $\|f_\nu\|_2\le |E_\nu|^{1/2}=1$ for all $\nu$, by the Banach-Alaoglu theorem, there is a weak-* convergent subsequence (which we just denote $(f_\nu)$) to a limit $f\in L^2$. Note that since weak-* convergence of $(f_\nu)$ to $f$ implies convergence as tempered distributions, it must be that $(\widehat{f_\nu})$ converge to $\widehat{f}$ as tempered distributions. Since $(\widehat{f_\nu})$ is a convergent sequence in $L^q$, it must therefore be true that $\widehat{f_\nu}\to \widehat{f}$ strongly in $L^q$.  


We claim that 
\[ \lim_{\nu,\mu\to\infty}\tfrac{1}{2}\|f_\nu+f_\mu\|_{\mc{L}}=1. \]
Indeed, $\|f_\nu+f_\mu\|_{\mc{L}}\le |E_\nu|^{1/p}+|E_\mu|^{1/p}=2$, so 
\[ {\bf{B}}_{q,d}=\lim_{\nu,\mu\to\infty}\|\tfrac{1}{2}\widehat{f_\nu}+\tfrac{1}{2}\widehat{f_\mu}\|_q\le \lim_{\nu,\mu\to\infty} \frac{\|\tfrac{1}{2}\widehat{f_\nu}+\tfrac{1}{2}\widehat{f_\mu}\|_q}{\|\tfrac{1}{2}f_\nu+\tfrac{1}{2}f_\mu\|_{\mc{L}}}\le{\bf{B}}_{q,d} \]
where we used Proposition \ref{Lorentz} in the final inequality. Also observe that 
\[ \lim_{\nu,\mu\to\infty}\|\tfrac{1}{2}f_\nu+\tfrac{1}{2}f_\mu\|_{\mc{L}}\le \|\tfrac{1}{2}1_{E_\nu}+\tfrac{1}{2}1_{E_\mu}\|_{\mc{L}}\le 1,    \]
so $\lim\limits_{\nu,\mu\to\infty}\|\frac{1}{2}1_{E_\nu}+\frac{1}{2}1_{E_\mu}\|_{\mc{L}}= 1$. By Lemma \ref{itcrowd}, since $\frac{1}{2}1_{E_\nu}+\frac{1}{2}1_{E_\mu}=1_{E_\nu\cap E_\mu}+\frac{1}{2}1_{E_\nu\Delta E_\mu}$, 
\[ \|\tfrac{1}{2}1_{E_\nu}+\tfrac{1}{2}1_{E_\mu}\|_{\mc{L}}=\tfrac{1}{2}|E_\nu\cap E_\mu|^{1/p}+\tfrac{1}{2}|E_\nu\cup E_\mu|^{1/2}.\]
Let $\delta_{\nu,\mu}>0$ be defined by $|E_\nu\cap E_\mu|=1-\delta_{\nu\mu} $, so $|E_\nu\cup E_\mu|=1+\delta_{\nu\mu}$. Since there exists $c>0$ so $(1-\delta)^{1/p}+(1+\delta)^{1/p}\le 2-c\delta^2$ for $|\delta|\le 1$, conclude that  $\lim\limits_{\nu,\mu\to\infty}|E_\nu\cap E_\mu|=1$. It follows that $|E_\nu\Delta E_\mu|\to 0$, so there exists a Lebesgue measurable set $E\subset \R^d$ such that $|E_\nu\Delta E|\to0$.

Note that for each $0<\eta<1$,
\begin{align*}
    \tfrac{1}{2}\|f_\nu+f_\mu\|_{\mc{L}}&\le \|1_{\{\tfrac{1}{2}|f_\nu+f_\mu|>1-\eta\}}+(1-\eta)1_{\{0<\tfrac{1}{2}|f_\nu+f_\mu|\le 1-\eta\}}\|_{\mc{L}} \\
    &= \eta|\{\tfrac{1}{2}|f_\nu+f_\mu|> 1-\eta\}|^{1/p}+(1-\eta)|\{0<\tfrac{1}{2}|f_\nu+f_\mu|\}|^{1/p}\\
    &\le  \eta|\{\tfrac{1}{2}|f_\nu+f_\mu|> 1-\eta\}|^{1/p}+(1-\eta)|E_\nu\cup E_\mu|^{1/p}.
\end{align*}
Since $\lim\limits_{\nu,\mu\to\infty}|E_\nu\cup E_\mu|=\lim\limits_{\nu,\mu\to\infty}\tfrac{1}{2}\|f_\nu+f_\mu\|_{\mc{L}}=1$, we conclude that 
\[\lim_{\nu,\mu\to\infty}|\{x:\tfrac{1}{2}|f_\nu(x)+f_\mu(x)|> 1-\eta\}|=1.\]
It follows that $\lim\limits_{\nu,\mu\to\infty}\|f_\nu+f_\mu\|_2=2$, and so $\lim\limits_{\nu,\mu\to\infty}\|f_\nu-f_\mu\|_p=0$ since $\|f_\nu-f_\mu\|_p\le 2\|f_\nu-f_\mu\|_2$
and by the parallelogram law,
\[ \|f_\nu-f_\mu\|_2^2+\|f_\nu+f_\mu\|_2^2=2(\|f_\nu\|_2^2+\|f_\mu\|_2^2). \]
Letting $\nu,\mu\to\infty$ gives the result.

\end{proof}

\begin{corollary} \label{precompactnesscor}Suppose that the claim in Question \ref{dual near-ext=slice} holds. Let $d\ge1$ and $q\in(2,\infty)$, $p=q'$. There exist a measurable function $f$ and a measurable subset $E$ of $\R^d$ with $|f|\le 1_E$ such that 
\[ {\bf{B}}_{q,d}=\frac{\|\widehat{f}\|_q}{|E|^{1/p}}=\frac{\|\widehat{f}\|_q}{\|f\|_{\mc{L}}}. \]
\end{corollary}

\begin{proof}[Proof of Corollary \ref{precompactnesscor}] By the proof of Theorem \ref{precompactness}, there exist a sequence of Lebesgue measurable subsets $E_\nu$ of $\R^d$, functions $f_\nu$ satisfying $|f_\nu|\le 1_{E_\nu}$ and $f,1_E\in L^p(\R^d)$ with $|f|= 1_E$ which satisfy $\lim_{\nu\to\infty}|E_\nu|^{-1/p}\|\widehat{f_\nu}\|_{q}={\bf{B}}_{q,d}$, $\lim_{\nu\to\infty}\|f_\nu-f\|_p=0$, and $\lim_{\nu\to\infty}|E_\nu\Delta E|=0$. It follows immediately that \[ \frac{\|\widehat{f}\|_q}{|E|^{1/p}}=\frac{\|\widehat{f}\|_q}{\|f\|_{\mc{L}}}={\bf{B}}_{q,d}. \]
\end{proof}

\begin{proof}[Proof of Corollary \ref{maincor}] By Lemma \ref{equivnorms}, Lemma \ref{extstruct}, and the inequality $\|g\|_{p1}\le q\|g\|_{p1}^*$ for all $g\in L(p,1)$ that 
\[ \sup_{\substack{g\in L(p,1)\\g\not=0}}\frac{\|\widehat{g}\|_q}{\|g\|_{p1}}\ge \frac{{\bf{B}}_{q,d}}{q}. \]
For the upper bound, we have a similar argument to the proof of Lemma \ref{extstruct}. Let $0\not=g\in L(p,1)$. Let $E=\{(y,s):|g(y)|>s \} $. Let $|g|e^{i\p}=g$ so we can use the layer cake representation
\[    g(x)=e^{i\p(x)}\int_0^\infty1_{E}(x,s)ds   .  \]
Then 
\begin{align*}  
\|\widehat{g}\|_q &= \left\|\left(\int_0^\infty e^{i\p(x)}1_{E}(x,s)ds\right)^{\widehat{\,}}\,\,\right\|_q = \left\| \int_0^\infty \widehat{e^{i\p}1_{E}}(\xi,s)ds\right\|_q  \\ 
&\le \int_0^\infty \|\widehat{e^{i\p}1_E}(\xi,s)\|_q ds \\
&\le \int_0^\infty {\bf{B}}_{q,d}\left|\{x:|g(x)|>s\}\right|^{1/p}ds \\
&=\frac{{\bf{B}}_{q,d}}{q} \int_0^\infty \|1_{\{x:|g(x)|>s\}}\|_{p1}ds  \\
&=\frac{{\bf{B}}_{q,d}}{q} \int_0^\infty\int_0^\infty t^{-1/q-1}\int_0^t1_{\{x:|g(x)|>s\}}^*(u)dudtds \\
&\le \frac{{\bf{B}}_{q,d}}{q}\int_0^\infty t^{-1/q-1}\int_0^tg^*(u)dudt= \frac{{\bf{B}}_{q,d}}{q}\|g\|_{p1} ,
\end{align*} 
so 
\[ \sup_{\substack{g\in L(p,1)\\g\not=0}}\frac{\|\widehat{g}\|_q}{\|g\|_{p1}}= \frac{{\bf{B}}_{q,d}}{q}. \]

Now if $0\not=f\in L(p,1)$ satisfies $\frac{\|\widehat{f}\|_q}{\|f\|_{p1}}=\frac{{\bf{B}}_{q,d}}{q}$, then by repeating the previous analysis, the above inequalities are equalities. Equality in the Minkowski integral inequality implies that for a.e. $(\xi,s)\in \R^d\times\R^+$,
\[ \widehat{e^{i\p}1_E}(\xi,s)=h(\xi)g(s)\]
for some measurable functions $h,g$. 
Since $e^{i\p}1_E(x,t)\in L^{2}$, in particular, $h$ and $\widecheck{h}$ in $L^2$. 
\[ 1_E(x,s)=e^{-i\p(x)}\widecheck{h}(x)g(s). \]
But then for every $(x,s)$ satisfying $|f(x)|>s$, we have 
\[ e^{-i\p(x)}\widecheck{h}(x)g(t)=1. \]
Suppose $|f(x)|>|f(y)|>0$. Then for all $0\le s<f(y)$, 
\[ e^{-i\p(x)}\widecheck{h}(x)=g(s)^{-1}=e^{i\p(y)}\widecheck{h(y)} ,\]
which is a contradiction unless $|f(x)|$ is constant on its support. Thus $f$ takes the form $ae^{i\p}1_S$ where $S\subset\R^d$ is a Lebesgue measurable subset and $a\in \R^+$.

For the existence of such an extremizer, by the proof of Theorem \ref{precompactness}, there exist a sequence of Lebesgue measurable subsets $E_\nu$ of $\R^d$, functions $f_\nu$ satisfying $|f_\nu|\le 1_{E_\nu}$, and $f,1_E\in L^p(\R^d)$ with $|f|= 1_E$ which satisfy $\lim_{\nu\to\infty}|E_\nu|^{-1/p}\|\widehat{f_\nu}\|_{q}={\bf{B}}_{q,d}$, $\lim_{\nu\to\infty}\|f_\nu-f\|_p=0$, and $\lim_{\nu\to\infty}|E_\nu\Delta E|=0$. Thus there exists $f\in L(p,1)$ satisfying
\[ \frac{\|\widehat{f}\|_q}{q|E|^{1/p}}=\frac{\|\widehat{f}\|_q}{\|f\|_{p1}}=\frac{{\bf{B}}_{q,d}}{q}. \]

\end{proof}

\section{Appendix}

\subsection{The Lorentz space $L(p,1)$}

We relate the three quasinorms on $L(p,1)$ defined in \textsection{\ref{Lorentzdiscussion}}. In the following lemma, we prove a formula for $\|s\|_{\mc{L}}$ where $s$ is a nonnegative simple function. 

\begin{lemma}\label{itcrowd} Let $d\ge 1$. Let $s=\sum_{n=1}^Na_n1_{A_n}$ where the $A_n$ are pairwise disjoint and of finite Lebesgue measure and $0<a_1<\cdots< a_N$. Let $a_0=0$ and let $B_n=\cup_{k=n}^NA_k$ for $n=1,\ldots,N$. Then 
\[ \|s\|_{\mc{L}}=\sum_{n=1}^N(a_n-a_{n-1})|B_n|^{1/p}. \]
\end{lemma}

\begin{proof} First we prove for any $k\ge 1$ that when $c_0=0<c_1<c_2<\cdots<c_k$ and $C_j=\cup_{i=j}^kE_i$ for measurable sets $E_i\subset\R^d$ of finite measure, 
\begin{align} \sum_{j=1}^k(c_j-c_{j-1})|C_j|^{1/p}\le \sum_{j=1}^kc_j|E_j|^{1/p}. \label{p,1Lineq}\end{align}

If  $k=1$, then clearly $\sum_{j=1}^{k}(c_j^n-c_{j-1}^n)|C_j^n|^{1/p}= c_1|C_1|^{1/p}=\sum_{j=1}^{1}c_j|E_j|^{1/p}$. Suppose for $k\ge 1$ that when $c_1<c_2<\cdots<c_k$ and $C_j=\cup_{i=j}^kE_i$ for measurable sets $E_i\subset\R^d$ of finite measure, 
\[ \sum_{j=1}^k(c_j-c_{j-1})|C_j|^{1/p}\le \sum_{j=1}^kc_j|E_j|^{1/p}. \]
Then 
\begin{align*} 
\sum_{j=1}^{k+1}(c_j-c_{j-1})|C_j^n|^{1/p} &= c_1|C_1|^{1/p}+(c_2-c_1)|C_2|^{1/p}+\cdots+(c_{k+1}-c_{k})|C_{k+1}|^{1/p} \\
&\le c_1(|E_1 |^{1/p}+|C_2 |^{1/p})+(c_2 -c_1 )|C_2 |^{1/p}+\cdots+(c_{k+1}-c_{k})|C_{k+1}|^{1/p} \\
&=c_1 |E_1 |^{1/p}+(c_2-c_0) |C_2 |^{1/p}+\cdots+(c_{k+1}-c_{k})|C_{k+1} |^{1/p} \\
&\le c_1|E_1 |^{1/p}+\sum_{j=2}^{k+1}c_j|E_j|^{1/p} =\sum_{j=1}^{k+1}c_j|E_j|^{1/p},
\end{align*}
so (\ref{p,1Lineq}) is proved.

Next we prove the lemma inductively, where notation is as in the statement of the lemma. If $N=1$, suppose $a_11_{A_1}=\sum_{j=1}^\infty b_j1_{S_j}$ where $b_j\ge 0$, $|S_j|<\infty$. Then 
\begin{align*}
a_1|A_1|^{1/p}=\|a_11_{A_1}\|_p&= \|\sum_{j=1}^Nb_j1_{S_j}+\sum_{j=N+1}^\infty b_j1_{S_j}\|_p \\
&\le \sum_{j=1}^\infty b_j|S_j|^{1/p}+\lim_{N\to\infty }\|\sum_{j=N+1}^\infty b_j1_{S_j}\|_p  
\end{align*}
where $ \lim_{N\to\infty }\|\sum_{j=N+1}^\infty b_j1_{S_j}\|_p =0$ by Lebesgue's dominated convergence theorem.

Now suppose that the lemma holds for $N-1\ge 1$. Consider 
\[ \sum_{n=1}^{N}a_n1_{A_n}=\sum_{j=1}^\infty b_j1_{S_j}\]

where $b_{j-1}\ge b_j\ge0$, $S_j\subset\cup_{n=1}^N A_n$, the $S_j$ are distinct, and $|S_j|>0$. From (\ref{p,1Lineq}), we have for each $M>0$ that 
\[ \sum_{j=1}^M(b_j-b_{j+1})|\cup_{k=1}^j S_{k}|^{1/p}\le \sum_{j=1}^Mb_j|S_j|^{1/p} .\]

Letting $M\to\infty $ and noting that $\sum_{j=1}^\infty b_j1_{S_j}=\sum_{j=1}^\infty (b_j-b_{j+1})1_{\cup_{k=1}^jS_k}$, we can assume that $S_1\subset S_2\subset\cdots $ and $b_j\ge 0$ but are not necessarily decreasing. Since the simple function $s$ achieves its $L^\infty$ norm on $A_N$, and the series takes its maximum on $S_1$, we must have $S_1=A_1$ and  

\[      a_N=\sum_{j=1}^\infty b_j,      \]

so $\sum_{k=j}^\infty b_k\to 0$ as $j\to\infty$. The simple function $s$ achieves its minimum (on a set of positive measure) in $\cup_{n=1}^NA_n$ on $A_1$, but the series takes the values of $\sum_{k=j}^\infty b_k$ on positive measure sets in $\cup_{n=1}^NA_n$, so there is no minimum unless $\sum_{k=j}^\infty b_k$ is zero for large enough $j$. Thus we may write 

\[ \sum_{n=1}^N(a_n-a_{n-1})1_{B_n}=\sum_{j=1}^Mb_j1_{S_j} \]

where $S_j\subset S_{j+1}$ and $b_j>0$. We note $B_1=S_M$ and $a_1=b_M$. Then invoking the inductive hypothesis, we have 
\[\sum_{n=2}^N(a_n-a_{n-1})|B_n|^{1/p}+a_1|B_1|^{1/p}\le \sum_{j=1}^{M-1}b_j|S_j|^{1/p}+b_M|S_M|^{1/p}, \]
as desired.

\end{proof}

For all $f\in L(p,1)$,
\begin{align} \label{coffee}  \|f\|_{p1}^*\le \|f\|_{p1}\le \frac{p}{p-1}\|f\|_{p1}^*,  \end{align}
which is proved in  Chapter V, \textsection{}3 in \cite{steinweiss}. From the nonincreasing property of $f^*$, it is clear that $\frac{1}{t}\int_0^tf^*(u)du\ge f^*(t)$ for $t>0$, which implies that$ \|f\|_{p1}\ge \|f\|_{p1}^*$. This combined with (\ref{coffee}) implies that $\|f\|_{p1}^*$ is finite if and only if $\|f\|_{p1}$ is finite.

\begin{lemma}\label{equivnorms} Let $d\ge 1$. Let $p>1$ and let $q$ be the conjugate exponent to $p$. For all measurable functions $f:\R^d\to\C$ with $\|f\|_{\mc{L}}<\infty$ and $\|f\|_{p1}^*<\infty$, $\|f\|_{\mc{L}}=\|f\|_{p1}^*$.
\end{lemma}

\begin{proof} First we show the equivalence for nonnegative simple functions. Write $s=\sum_{n=1}^Na_n1_{A_n}$ where the $A_n$ are pairwise disjoint and $0<a_1<\cdots< a_N$. Let $a_0=0$ and let $B_n=\cup_{k=n}^NA_k$ for $n=1,\ldots,N$, and let $|B_{N+1}|=0$. 

Calculate 

\begin{align*}
\|s\|_{p1}^*&= \frac{1}{p}\int_0^\infty t^{-1/q}s^*(t)dt =\frac{1}{p}\sum_{n=0}^{N-1}\int_{|B_{N-n+1}|}^{|B_{N-n}|} t^{-1/q}\inf\{r:|\{x:|s(x)|>r\}|\le t\}dt \\
&=\frac{1}{p}\sum_{n=0}^{N-1}\int_{|B_{N-n+1}|}^{|B_{N-n}|} t^{-1/q}\inf\{r:|\{x:|s(x)|>r\}|\le |B_{N-n+1}|\}dt \\
&=\frac{1}{p}\sum_{n=0}^{N-1}a_{N-n}\int_{|B_{N-n+1}|}^{|B_{N-n}|} t^{-1/q}dt \\
&= \sum_{n=0}^{N-1}a_{N-n}\left(|B_{N-n}|^{1/p}-|B_{N-n+1}|^{1/p}\right) \\
&= \sum_{n=0}^{N-1}a_{N-n}|B_{N-n}|^{1/p}- \sum_{n=0}^{N-1}a_{N-n}|B_{N-n+1}|^{1/p} \\
&= \sum_{n=1}^{N}a_{n}|B_{n}|^{1/p}- \sum_{n=1}^{N}a_{n-1}|B_{n}|^{1/p}= \sum_{n=1}^{N}(a_{n}-a_{n-1})|B_{n}|^{1/p}.
\end{align*}

Thus by Lemma \ref{itcrowd}, we have $\|s\|_{\mc{L}}=\|s\|_{p1}^*$ for all nonnegative simple functions. 

Next, consider a $f\in L(p,1)$ with finite support $A$ and $L^\infty$ norm $M>0$. From the definition of $\|\cdot \|_{\mc{L}}$ and Lemma \ref{itcrowd}, we can choose nonnegative simple functions $|f|-1/n\le s_n\le |f|$ such that $\lim_{n\to\infty}s_n(x)=|f(x)|$ for a.e. $x\in\R^d$ and 
\[ \|f\|_{\mc{L}}=\lim_{n\to\infty}\|s_n\|_{\mc{L}}=\lim_{n\to\infty}\|s_n\|_{p1}^*.  \]

Note that 
\begin{align*}
    \frac{1}{p}\int_0^\infty t^{-1/q}s_n^*(t)dt &=\frac{1}{p}\int_0^{|A|} t^{-1/q}s_n^*(t)dt. 
\end{align*}
Since $(|f|-1/n1_A)^*\le s_n^*\le |f|^*$, we have the upper bound 
\begin{align*}
    \frac{1}{p}\int_0^{|A|} t^{-1/q}s_n^*(t)dt\le \frac{1}{p}\int_0^{|A|} t^{-1/q}f^*(t)dt
\end{align*}
and the lower bound 
\begin{align*} 
\frac{1}{p}\int_0^{|A|} t^{-1/q}\inf\{r:|\{x:s_n(x)>r\}|\le t\}dt&\ge  \frac{1}{p}\int_0^{|A|} t^{-1/q}\inf\{r:|\{x:|f(x)|>r+1/n\}|\le t\}dt\\
&= \frac{1}{p}\int_0^{|A|} t^{-1/q}\inf\{r:|\{x:|f(x)|>r\}|\le t\}dt-p|A|^{1/p}\frac{1}{n}. 
\end{align*} 
Thus by the squeeze theorem, we have that $\lim_{n\to\infty }\|s_n\|_{p1}^*=\|f\|_{p1}^*$.

For general $f\in L(p,1)$, define $f_n=f1_{\{1/n\le |f|\le n\}}$. We argue that $\|f\|_{\mc{L}}=\lim_{n\to\infty}\|f_n\|_{\mc{L}}$. 

If $|f_n|= \sum_na_n1_{A_n}$, $|f|1_{\{|f|>n\}}=\sum_mb_m1_{B_m}$, and $|f|1_{\{|f|<1/n\}}=\sum_k c_k1_{C_k}$, then $|f|=\sum_n a_n 1_{A_n}+\sum_mb_m1_{B_m}+\sum_kc_k1_{C_k}$ and so
\[ \|f\|_{\mc{L}}\le \|f_n\|_{\mc{L}}+\||f|1_{\{|f|>n\}}\|_{\mc{L}}+\||f|1_{|f|<1/n\}}\|_{\mc{L}} . \]

Since if $|f|=\sum_me_m1_{E_m}$, $e_m>0$ with $\sum_me_m|E_m|^{1/p}<\infty$ then $|f_n|=\sum_me_m1_{E_m\cap \{1/n\le |f|\le n\}}$ with 
\[ \sum_m e_m|E_m\cap\{1/n\le |f|\le n\}|^{1/p}\le \sum_me_m|E_m|^{1/p}<\infty, \]
we also have $\|f_n\|_{\mc{L}}\le \|f\|_{\mc{L}}$. To show that $\|f\|_{\mc{L}}=\lim_{n\to\infty}\|f_n\|_{\mc{L}}$, it suffices to show that $\lim_{n\to\infty}\|f1_{\{|f|>n\}}\|_{\mc{L}}=\lim_{n\to\infty}\|f1_{\{|f|>1/n\}}\|_{\mc{L}}=0$.

If $|f|=\sum_me_m1_{E_m}$, $e_m>0$ with $\sum_m e_m|E_m|^{1/p}<\infty$, then 
\[ \limsup_{n\to\infty} \||f|1_{\{|f|>n\}}\|_{\mc{L}}\le  \limsup_{n\to\infty}\sum_me_m|E_m\cap \{|f|>n\}|^{1/p}=0 \]
where we used the monotone convergence theorem in the last line. Similarly, we have that 
\[ \limsup_{n\to\infty}\|f1_{\{|f|<1/n\}}\|_{\mc{L}}\le \limsup_{n\to\infty}\sum_m e_m|E_m\cap \{|f|<1/n\}|^{1/p}.  \]
Since $\sum_me_m|E_m\cap \{|f|<1/n\}|^{1/p}$ is a decreasing sequence in $n$, 
\[ \limsup_{n\to\infty}\sum_me_m|E_m\cap \{|f|<1/n\}|^{1/p}=\inf_n\sum_me_m|E_m\cap \{|f|<1/n\}|^{1/p}. \]

We also have for each $M>1$
\begin{align*} 
\inf_{n}\sum_me_m|E_m\cap \{|f|<1/n\}|^{1/p}&\le \inf_n\sum_{m\le M}e_m|E_m\cap \{|f|<1/n\}|^{1/p}+\sum_{m>M}e_m|E_m|^{1/p}\\
&= \sum_{m\le M}e_m|E_m\cap \{|f|=0\}|^{1/p}+\sum_{m>M}e_m|E_m|^{1/p}\\
&=\sum_{m>M}e_m|E_m|^{1/p}. 
\end{align*} 
Letting $M$ go to infinity, we have $\lim_{n\to\infty}\|f1_{\{|f|<1/n\}}\|_{\mc{L}}=0$. Conclude that $\|f\|_{\mc{L}}=\lim_{n\to\infty}\|f_n\|_{\mc{L}}=\lim_{n\to\infty}\|f_n\|_{p1}^*$. 

Finally, we need to show that $\lim_{n\to\infty}\|f_n\|_{p1}^*=\|f\|_{p1}^*$. Since $\|f_n\|_{p1}^*\le \|f\|_{p1}^*$ for each $n$ and $\lim_{M\to\infty }\int_M^\infty t^{-1/q}f_n^*(t)dt\le \lim_{M\to\infty}\int_M^\infty t^{-1/q}f^*(t)dt=0$, it suffices to show that for each $M>0$,
\[ \lim_{n\to\infty }\int_0^Mt^{-1/q}f_n^*(t)dt\ge \int_0^Mt^{-1/q}f^*(t)dt. \]

We note that 
\begin{align*}
\int_0^Mt^{-1/q}f_n^*(t)dt&= \int_0^Mt^{-1/q}\inf\{r:|\{x:|f_n(x)|>r\}|\le t\} dt\\
    &\ge \int_0^Mt^{-1/q}\inf\{r:|\{x:|f(x)|1_{\{|f|\le n\}}>r\}|\le t\} dt-pM^{1/p}\frac{1}{n}\\
    &= \ge \int_0^Mt^{-1/q}f^*(t+|\{x:|f(x)|>n\}|) dt-pM^{1/p}\frac{1}{n}. 
\end{align*}
Since $\lim_{n\to\infty }|\{x:|f(x)|>n\}|=0$ and $f^*$ is a.e. continuous, by the Lebesgue dominated convergence theorem, 
\[\lim_{n\to\infty}\int_0^Mt^{-1/q}f_n^*(t)dt\ge \int_0^Mt^{-1/q}f^*(t)dt.  \]

\end{proof}

\begin{corollary}\label{equivnorms2} Let $d\ge 1$. Let $p>1$ and let $q$ be the conjugate exponent to $p$. For all measurable functions $f:\R^d\to\C$, $\|f\|_{\mc{L}}<\infty$ if and only if  $\|f\|_{p1}^*<\infty$.
\end{corollary}

\begin{proof} Suppose that $\|f\|_{\mc{L}}<\infty$. We showed in the proof of Lemma \ref{equivnorms} that $\|f\|_{\mc{L}}=\lim_{n\to\infty}\|f_n\|_{\mc{L}}$ where $|f_n|$ are bounded with finite support and monotonically increasing a.e. to $|f|$. We also showed that for those $f_n$, $\lim_{n\to\infty}\|f_n\|_{p1}^*=\|f\|_{p1}^*$, so $\|f\|_{p1}^*<\infty$. 

Next, suppose $\|f\|_{p1}^*<\infty$. By definition of $\|f\|_{\mc{L}}$ (regardless of whether this quantity is finite or infinite), there exist simple functions $0\le s_n\le |f|$ such that $\|f\|_{\mc{L}}=\lim_{n\to\infty}\|s_n\|_{\mc{L}}$. But we showed in the proof of Lemma \ref{equivnorms} that $\|s_n\|_{\mc{L}}=\|s_n\|_{p1}^*$ for each $n$. Since $\|s_n\|_{p1}^*\le \|f\|_{p1}^*$ for all $n$, we must have $\|f\|_{\mc{L}}<\infty$.  

\end{proof}

\end{document}